\documentclass[11pt, reqno]{amsart}
\usepackage{fullpage}
\usepackage[mathscr]{eucal}
\usepackage[all]{xy}
\usepackage{mathrsfs}
\usepackage{xypic}
\usepackage{amsfonts}
\usepackage{amsmath}
\usepackage{amsthm}
\usepackage{amssymb}
\usepackage{latexsym}
\usepackage{eufrak}

\theoremstyle{plain}
\newtheorem{thm}[equation]{Theorem}
\newtheorem{prop}[equation]{Proposition}
\newtheorem{lem}[equation]{Lemma}
\newtheorem{cor}[equation]{Corollary}
\theoremstyle{definition}
\newtheorem{defn}[equation]{Definition}

\theoremstyle{remark}

\newtheorem{rem}[equation]{Remark}
\newtheorem{remark}[equation]{Remark}
\newtheorem{claim}[equation]{Claim}
\newtheorem{notation}[equation]{Notation}

\makeatletter
\renewcommand{\subsection}{\@startsection{subsection}{2}{0pt}{-3ex
plus -1ex minus -0.2ex}{-2mm plus -0pt minus
-2pt}{\normalfont\bfseries}} \makeatother

\numberwithin{equation}{section}
\numberwithin{equation}{subsection}

\newcommand{\erem}{\hphantom{.}\hfill$\lozenge$\end{rem}}

\newcommand{\bimod}[1]{#1\text{-}{\sf{bimod}}}

\newcommand{\br}{_{\natural}}
\newcommand{\hdot}{{\:\raisebox{2pt}{\text{\circle*{1.5}}}}}
\newcommand{\idot}{{\:\raisebox{2pt}{\text{\circle*{1.5}}}}}

\DeclareMathOperator{\Ext}{\mathrm{Ext}}

\DeclareMathOperator{\rk}{\mathrm{rk}}

\DeclareMathOperator{\Lie}{\mathrm{Lie}}
\DeclareMathOperator{\Tr}{\mathrm{Tr}}
\DeclareMathOperator{\Rep}{\mathrm{Rep}}
\DeclareMathOperator{\Ad}{\mathrm{Ad}}
\DeclareMathOperator{\act}{\mathrm{act}}

\def\map{\longrightarrow}
\newcommand{\dis}{\displaystyle}

\DeclareMathOperator{\Spec}{\mathrm{Spec}}
\DeclareMathOperator{\pr}{pr}
\newcommand{\iso}{{\;\stackrel{_\sim}{\to}\;}}
\newcommand{\cd}{\!\cdot\!}

\DeclareMathOperator{\GL}{\mathrm{GL}}

\def\ccirc{{{}_{\,{}^{^\circ}}}}

\newcommand{\ggamma}{{\boldsymbol{\gamma}}}
\newcommand{\g}[1]{\mathfrak{#1}}
\newcommand{\scr}[1]{\mathscr{#1}}

\newcommand{\ddr}{ d_{_\text{DR}}}

\DeclareMathOperator{\End}{\mathrm{End}}
\newcommand{\oEnd}{\mathrm{End}}
\newcommand{\oba}{\Om^B\!A }
\newcommand{\boba}{\bom^B\!A }
\DeclareMathOperator{\Hom}{\mathrm{Hom}}

\DeclareMathOperator{\Der}{\mathrm{Der}}
\DeclareMathOperator{\dder}{{\mathbb{D}er}}

\newcommand{\arr}{\overset{{\,}_\to}}

\renewcommand{\b}{{\mathsf b}}
\renewcommand{\d}{{\mathsf d}}
\newcommand{\exact}{_{\op{exact}}}
\newcommand{\closed}{_{\op{closed}}}

\newcommand{\De}{\Delta}

\newcommand{\id}{\mathrm{id}}
\newcommand{\Id}{\mathrm{Id}}
\newcommand{\op}{\operatorname}

\newcommand{\Mat}{\mathrm{Mat}}

\newcommand{\fr}{{\mathfrak m}}
\newcommand{\im}{\mathrm{Im}}

\newcommand{\wt}{\widehat }
\newcommand{\La}{{\mathsf{\Lambda}}}

\newcommand{\bc}{{\mathbf{c}}}
\newcommand{\cyclic}{_{\text{cyc}}}
\newcommand{\cyc}{{\text{cyc}}}
\renewcommand{\a}{\!A}

\newcommand{\bas}{_{\operatorname{basic}}}
\newcommand{\bi}{\imath}
\newcommand{\bl}{{\scr L}}
\newcommand{\oma}{{\Om^1\a}}
\newcommand{\com}{_{\text{\footnotesize comm}}}
\newcommand{\omb}{{\Om^\hdot\a}}
\newcommand{\omh}{{\Om^\hdot}}

\newcommand{\rep}{{\Rep(A,V)}}

\DeclareMathOperator{\Aut}{\mathrm{Aut}}
\newcommand{\bom}{{\overline{\Omega}}}
\newcommand{\bomt}{\bom \hat \otimes \k[t,t^{-1}]}
\newcommand{\ttc}{\hat \otimes \k[t,t^{-1}]}

\renewcommand{\o}{{\otimes}}
\newcommand{\be}{\beta}

\DeclareMathOperator{\ad}{\mathrm{ad}}

\def\bplus{{\mbox{$\bigoplus$}}}

\DeclareMathOperator{\Ker}{\mathrm{Ker}}
\DeclareMathOperator{\ev}{{\mathrm{ev}}}

\def\calT{{\mathscr{T}}}

\newcommand{\evom}{\ev_{_\Om}}
\newcommand{\BL}{{\mathbf{L}}}
\newcommand{\bbi}{{\mathbf i}}

\renewcommand{\lll}{{[\![}}
\newcommand{\rrr}{{]\!]}}

\newcommand{\mto}{\mapsto}
\renewcommand{\lto}{\longmapsto}
\newcommand{\triv}{{\mathsf{triv}}}
\newcommand{\inv}{^{-1}}
\newcommand{\vi}{${\en\sf {(i)}}\;$}
\newcommand{\vii}{${\;\sf {(ii)}}\;$}
\newcommand{\viii}{${\sf {(iii)}}\;$}

\newcommand{\sset}{\subset}

\newcommand{\into}{{}^{\,}\hookrightarrow^{\,}}
\newcommand{\too}{\,\longrightarrow\,}
\newcommand{\onto}{\twoheadrightarrow}
\newcommand{\tooo}{{\;{-\!\!\!-\!\!\!-\!\!\!-\!\!\!\longrightarrow}\;}}

\newcommand{\oper}{\operatorname}

\newcommand{\R}{\mathsf R}

\newcommand{\Z}{\mathbb Z}
\newcommand{\Q}{\mathbb Q}

\newcommand{\DR}{{\oper{DR}}}
\newcommand{\starb}{\star_{_\beta}}

\renewcommand{\th}{\theta}

\newcommand{\munc}{{\mu_{\mathrm{nc}}}}

\newcommand{\Om}{\Omega}

\newcommand{\Th}{\Theta}

\DeclareMathOperator{\gr}{\mathrm{gr}}

\def\ip<#1,#2>{\left\langle#1,#2\right\rangle}
\def\sp<#1>{\left\langle#1\right\rangle}

\newcommand{\Triv}{\mathrm{Triv}}

\newcommand{\ch}{\mathrm{ch}}

\newcommand{\eu}{\mathsf{eu}}

\def\ip<#1,#2>{\left\langle#1,#2\right\rangle}

\def\npb{\noindent$\bullet\quad$\parbox[t]{158mm}}

\newcommand{\om}{\omega}

\newcommand{\ka}{\kappa}

\newcommand{\al}{{\alpha}}
\newcommand{\A}{{\mathscr A}}

\newcommand{\en}{{\enspace}}

\newcommand{\wh}{\widehat}

\newcommand\ldp{(}  \newcommand\rdp{)}

\def\LL{{\mathbf{L}}}
\def\D{{\mathsf{D}}}

\def\B{{\sf B}}

\def\k{{\Bbbk}}
\def\s{{\mathbb{S}}}

\begin{document}

\title{Free products, cyclic homology,  and \\
\vskip2pt
the Gauss-Manin connection} 
\author{\small{Victor Ginzburg and Travis Schedler}}
\maketitle
\vskip -5mm
\centerline{\textit{with an Appendix by}}
\vskip2pt
\centerline{\sc {\large{Boris Tsygan}}}

\begin{abstract} We use the techniques of Cuntz and Quillen to present
  a new approach to periodic cyclic homology.  Our construction is
  based on $(\dis(\omb)[t], \d+t\cdot\bi_\De)$, a {\em noncommutative
    equivariant de Rham complex} of an associative algebra $A$.  Here
  $\d$ is the Karoubi-de Rham differential and $\bi_\Delta$ is an
  operation analogous to contraction with a vector field.  As a
  byproduct, we give a simple explicit construction of the Gauss-Manin
  connection, introduced earlier by E. Getzler, on the relative
  periodic cyclic homology of a flat family of associative algebras
  over a central base ring.

We introduce and study  \textit{free-product
deformations} of an associative algebra, a new type of deformation over
a not necessarily commutative base ring.
Natural examples of free-product
deformations arise from  preprojective algebras and group algebras for compact
surface groups.
\end{abstract}

\tableofcontents

\section{Introduction}
Throughout, we fix a field $\k$ of characteristic 0 and write
$\o=\o_\k$ (although, up to passing to reduced versions of homology
theories, everything generalizes to the case that $\k$ is an arbitrary
commutative ring containing $\Q$, provided the algebra $A$ has the
property that $\k \to A$ is a $\k$-split central injection). By an
algebra we always mean an associative unital $\k$-algebra, unless
explicitly stated otherwise.  Given an algebra $A$, we view the space
$A\otimes A$ as an $A$-bimodule with respect to the {\em outer}
bimodule structure, which is defined by the formula $b(a'\otimes
a'')c:=(ba')\otimes (a''c)$, for any $a',a'',b,c\in A$. By an
$A$-bimodule, we will always mean an $(A \otimes
A^{\text{op}})$-module, i.e., an $A$-bimodule on which the left and
right action of $\k$ coincide.

\subsection{Double derivations}
It is well-known that a regular vector field on a smooth affine
algebraic variety $X$ is the same thing as a derivation 
$\k[X]\to \k[X]$ of the coordinate ring of $X$. Thus,
derivations of a commutative algebra  play the role
of vector fields.

It has been commonly accepted until recently that this
 point of view 
 applies to noncommutative algebras  as well.
A first indication towards a different point of view
was a discovery by Crawley-Boevey \cite{CB} that,
for a smooth affine
curve $X$ with coordinate ring $A=\k[X]$,
the algebra of differential operators on $X$
 can be constructed by means of {\em double derivations}
$A\to A\o A$, rather than ordinary derivations
$A\to A$. Since then, the significance
of double derivations in noncommutative geometry
was  explored further in \cite{VdB} and \cite{CBEG}.

To explain the role of  double derivations in more detail
we first recall 
 some basic definitions.

\subsection{(Double) derivations as infinitesimal automorphisms}\label{dder}
Recall that a free product of two  algebras ${A}$ and $B$, 
is an associative algebra ${A}*B$ that contains $A$ and $B$ as
subalgebras
and
 whose elements are formal $\k$-linear combinations of words
$a_1b_1a_2b_2\ldots a_nb_n$, for any $n\geq 1$ and
$a_1,\ldots,a_n\in {A},\,b_1,\ldots,b_n\in B$.
These words are taken up to the equivalence  imposed
by the relation $1_{A}=1_{B}$; for instance, 
$\ldots b1_{A} b'\ldots=\ldots b1_{B}b'\ldots=\ldots
(b\cdot b')\ldots$,
for any $b,b'\in B$. 

 Let  $N$ be an $A$-{\em bimodule}. 
A  $\k$-linear map $ f: A\to N$
is said to be a derivation of $A$ with coefficients in $N$
if  $ f(a_1a_2)= f(a_1)a_2+a_1 f(a_2),\,
\forall a_1,a_2\in A$.
Given  a subalgebra $R\sset A$,
 we let $\Der_R(A,N)$ denote the space of $R$-linear derivations 
of $A$ with respect to the subalgebra $R$, that is,
of derivations $A\to N$ that annihilate the subalgebra $R$.

Derivations of an algebra $A$ may 
be viewed as `infinitesimal automorphisms'. 
Specifically,  let $A[t]=A\o\k[t]$ be the
polynomial ring in one variable with coefficients in $A$.
The natural algebra embedding
$A\into A[t]$ makes $A[t]$ an $A$-bimodule.

A well-known elementary calculation yields the following.

\begin{lem}\label{[t]}  The following properties of a $\k$-linear map  $\th: A\to A$
are equivalent:

\npb{the map  $\th$ is a derivation of the algebra $A$;}

\npb{the map $A\to t\cdot A[t],\,a\mapsto t\cdot \th(a)$
is a derivation of the algebra $A$ with coefficients in $t\cdot A[t]$;}

\npb{the map $a\mto a+ t\cdot \th(a)$ gives
an algebra homomorphism
$ A\to A[t]/t^2\cd A[t]$.}
\end{lem}

All the above holds true, of course, no matter whether
the algebra $A$ is commutative or not.
Yet, the element $t$, the formal parameter,
 is by definition a {\em central} element
of the algebra  ~$A[t]$.  
In noncommutative geometry, 
the assumption that the formal parameter $t$ be central
is not quite natural, however.
Put another way, while the tensor product is a coproduct
in the category of commutative associative algebras, the free product
is  a coproduct
in the category of not necessarily  commutative associative algebras.

We see that, in noncommutative geometry,
the algebra $A_t=A*\k[t]$,
freely generated by $A$ and an indeterminate $t$,
should play the role of  the polynomial algebra 
$A[t]$.
We are going to argue that, once the polynomial algebra
$A[t]$
is replaced by the algebra $A_t$,
it becomes more natural to replace
derivations $A\to A$ by  {\em double derivations}, i.e., by derivations 
$A\to A\o A$ where $A\o A$ is viewed as an $A$-bimodule with respect to
the outer bimodule structure.
To see this, write $A_t^+=A_t\cdot t\cdot A_t$,
a two-sided ideal generated by the element $t$.
Then, there are natural $A$-bimodule isomorphisms
\begin{equation}\label{AtA}
A_t/A_t^+\iso A,
\quad\text{and}\quad A_t/(A^+_t)^2\,\iso\,
A\oplus (A\o A),
\quad  a+ a'\,t\,a''\lto a \oplus (a'\o a'').
\end{equation} 

Given  a $\k$-linear map $\Th: A\to A\o A$ we will  use symbolic 
Sweedler notation to write this map
as $a\mto \Th'(a)\otimes\Th''(a)$,
where we systematically
suppress the summation symbol.

Now, a free product analogue of Lemma~\ref{[t]} reads as follows.

\begin{lem}\label{del_der}  The following properties of a $\k$-linear map  
 $\Th: A\to A\o A$
are equivalent:

\npb{the map $\Th$ is a double derivation;}

\npb{the map $A\to A^+_t,\
a\mto \Th'(a)\,t\,\Th''(a)$, is a  derivation of the algebra $A$ with coefficients
in the $A$-bimodule $A^+_t$;}

\npb{the map
$a\mto a+\Th'(a)\,t\,\Th''(a)$
gives an algebra homomorphism
$ A\to A_t/(A_t^+)^2
$.}
\end{lem}

\subsection{Layout of the paper} In \S \ref{extkadrcplx}, we recall the
definition of the DG algebra of noncommutative differential forms
\cite{Co,CC}, following \cite{CQ1}, and that of the Karoubi-de Rham
complex \cite{Ka}.  We also introduce an {\em extended Karoubi-de
  Rham complex},
that will play a crucial role later.  In \S \ref{noncomm_calc}, we
develop the basics of noncommutative calculus involving the action of
double derivations on the extended Karoubi-de Rham complex, via Lie
derivative and contraction operations.

In \S 4, we state three main results of the paper.  The first two,
Theorem \ref{hoch_thm} and Theorem \ref{main}, provide a description,
in terms of the Karoubi-de Rham complex, of the Hochschild homology of
an algebra $A$ and of the periodic cyclic homology of $A$,
respectively.  The third result, Theorem \ref{gmthm}, gives a formula
for the Gauss-Manin connection on periodic cyclic homology of a family
of algebras, \cite{Getz}, in a way that avoids complicated formulas
and resembles equivariant cohomology.  These results are proved in \S
\ref{pf_main}, using properties of the Karoubi operator and the
harmonic decomposition of noncommutative differential forms introduced
by Cuntz and Quillen, \cite{CQ1, CQ2}.

In \S \ref{repf}, we establish a connection between cyclic homology and
equivariant cohomology via the representation functor.  More
precisely, we give a homomorphism from our noncommutative equivariant
de Rham
complex (which extends the complex used to compute cyclic homology) to
the equivariant de Rham complex computing equivariant cohomology
of the representation variety.

In \S \ref{ncds}, we introduce a new  notion of free product
deformation over a not necessarily commutative base. We extend  classic
results of Gerstenhaber
concerning deformations of an associative algebra $A$ to our new setting
 of free product
deformations. To this end, we consider a double-graded Hochschild 
complex $\big(\oplus_{p, k\geq 2}\ C^p(A, A^{\o k}),\ \b\big)$. We define a new
associative product $f,g\mto f\vee g$ on that complex, and study
Maurer-Cartan equations of the form $\b(\be)+\frac{1}{2}\be\vee\be=0$.

\subsection{Acknowledgements}{\small{We 
thank   Yan Soibelman, Boris Tsygan,  and Michel Van den Bergh
for useful comments.
The first author  was partially supported by the NSF grant
DMS-0303465 and CRDF grant RM1-2545-MO-03. The second author
was partially supported by an NSF GRF.}}

\section{Noncommutative calculus} \label{noncomm_calc}
\subsection{The commutator quotient} 
Let $B=\oplus_{k\in \Z} B^k$ be a $\Z$-graded
algebra and $M=\oplus_{k\in \Z} M^k$ a  $\Z$-graded $B$-bimodule.
Given a homogeneous element $u\in B^k$ or $u\in M^k$,
we put $|u|:=k$. 
A linear map $f: B^\hdot\to M^{\hdot+n}$ is said
to be a {\em degree} $n$  {\em graded} derivation if, for any
 homogeneous $u,v\in B$, 
$f(uv)=f(u)\cdot v+ (-1)^{n|u|}u\cdot f(v)$. 
Given a graded subalgebra $R\sset B$,
we let $\Der^n_R(B,M)$ denote the vector
space of degree $n$ graded $R$-linear derivations. 
The direct sum 
$\Der^\hdot_R\! B:=
\bigoplus_{n\in \Z}\Der^n_R(B,B)$ has a natural Lie {\em super}-algebra structure
given by the {\em super}-commutator.

Let $[B,B]$ 
be 
the $\k$-linear span of the
set  $\{b_1b_2-(-1)^{pq} b_2b_1$,
$b_1\in B^p,b_2\in B^q,\,p,q\in\Z\}$. 
We put 
$B_\cyc:=B/[B,B]$.
Any (ungraded) algebra may be regarded as a graded
algebra concentrated in degree zero. Thus, for an
algebra $B$ without grading, we have the subspace
$[B,B]\sset B$ spanned by ordinary commutators
and the corresponding commutator quotient space
$B_\cyc=B/[B,B]$.

We write
 $T_B M=\oplus_{n\geq 0}\ T^n_BM$ for the tensor algebra of 
a $B$-bimodule $M$. Thus, $T_B^\hdot M$ is a graded associative
algebra with $T^0_BM=B$.

\subsection{Noncommutative differential forms}
Fix an algebra $A$ and a subalgebra $R\sset A$.  (After this section,
we will only take $R$ to be either $\k$ or $\k[t]$, but other
interesting examples include when $R = \k^I$ and $A$ is a quotient of
the path algebra of a quiver with vertex set $I$.) Let
$\Om^1_RA:=\Ker(m)$ be the kernel of the multiplication map $m:
A\otimes_R A\to A$, and write $i_\Delta: \Om^1_RA\into A\otimes_R A$
for the tautological embedding. Thus, $\Om^1_RA$ is a $A$-bimodule,
called the bimodule of {\em noncommutative one-forms} on the algebra $A$
relative to the subalgebra $R$.  One has a short exact sequence of
$A$-bimodules, see \cite[\S 2]{CQ1},
\begin{equation}\label{fund}
0\map\Om^1_RA\stackrel{i_\De}\too A\o_R A\stackrel{m}\too A\map 0.
\end{equation}

The assignment $a\mapsto \d a:=1\o a-a\o 1$ gives
 a canonical
derivation $\d:A\to \Om^1_RA$. This derivation is `universal'
in the sense
that, for any $A$-bimodule $M$, there is a canonical bijection
\begin{equation}\label{der_def}
\Der_R(A,M)\iso\Hom_{\bimod{A}}(\Om^1_RA, M),\quad
\th\mapsto i_\th,
\end{equation}
where  $i_\th:\Om^1_RA\to M$ stands for a $A$-bimodule map defined by
the formula $i_\th(u\,\d v):=u\cdot\th(v)$.

The map $\d$ extends uniquely to a degree 1 derivation
of $\Om_R A:=T_A^\hdot(\Om^1_RA)$, the tensor algebra
 of the 
$A$-bimodule  $\Om^1_RA$.
Thus,  $(\Om^\hdot_RA ,\d)$ is a DG algebra called
the  algebra of noncommutative  differential forms on $A$
relative to the subalgebra $R$
(we will interchangeably use the notation $\Om_R A$ or
$\Om^\hdot_RA $ depending on whether we want to emphasize the grading or
not).
For each $n\geq 1$, there is a standard
isomorphism of left $A$-modules,
see \cite{CQ1}, $\Om^n_RA=A\otimes_R T_R^n(A/R)$; usually, one writes
$a_0\,\d a_1\,\d a_2\ldots \d a_n\in \Om^n_RA$ for the $n$-form
corresponding to an element $a_0\o(a_1\o\ldots\o a_n)\in
A\otimes_R T_R^n(A/R)$ under this isomorphism.
The de Rham differential  $\d: \Om^\hdot_RA \to\Om^{\hdot+1}_RA$
is given by the formula
$\d: a_0\,\d a_1\,\d a_2\ldots \d a_n\mapsto
\d a_0\,\d a_1\,\d a_2\ldots \d a_n$.

Following Karoubi \cite{Ka}, we
define the (relative) noncommutative de Rham complex of
$A$ as
$$\DR_R  A:=(\Om_R A)_\cyc=\Om_R A/[\Om_R A,\Om_R A],$$
 the super-commutator quotient  of the
{\em graded} algebra $\Om^\hdot_RA $. The space $\DR_R A$ comes
equipped with a natural grading and
with the de Rham differential $\d: \DR^\hdot_R A\to \DR_R^{\hdot+1}A$,
induced from the one on $\Om^\hdot_RA $.
In degrees 0 and 1, $\DR^0_RA=A/[A,A]$ and
$\DR^1_RA=\Om^1_RA/[A,\,\Om^1_RA]$, respectively.

In the `absolute' case $R=\k$ we will use unadorned notation
$\Om^nA:=$ $\Om^n_\k A,
\ \DR A:=\DR_\k A$, etc.

The DG algebra $(\Om^\hdot_RA,\,\d)$ can be characterized by a universal
 property
saying that it is the universal DG $R$-algebra
generated by $A$ (see \cite[Corollary 2.2]{CQ1}).
The universal
property implies, in particular, for any algebras $A$ and $B$,
a canonical isomorphism 
\begin{equation}\label{AB}
\Om(A*B)\ \cong\ \Om A\, *\, \Om B.
\end{equation}

\subsection{Lie derivative and contraction for noncommutative differential forms} In this section, we introduce  operations
of Lie derivative and contraction on noncommutative 
differential forms.

Fix an algebra $A$ and a subalgebra $R\sset A$.
Any derivation $\th\in\Der_R A$ gives rise, naturally, to 
a pair of graded derivations of the
 graded algebra $\Om^\hdot_RA$, a contraction operation
$i_\th\in \Der_A^{\hdot-1}(\Om^\hdot_RA)$ and a  Lie derivative operation
$L_\th\in  \Der_R^\hdot(\Om^\hdot_RA)$, respectively.
To define these operations it is convenient to use the
following construction.

Let $K$ be a graded $\Om^\hdot_RA$-bimodule 
and  $f: \Om^1_RA\to K$ an $A$-bimodule map.
One shows, by
adapting  the proof of \cite[Proposition 2.6]{CQ1}
to a graded setting, that the assignment
$a \mto 0,\ \d a\mto f(\d a)$ can be extended uniquely
to an $A$-linear derivation $T^\hdot_A(\Om^1_RA) \to K$.
Hence, one has the following chain of canonical isomorphisms
\begin{equation}\label{K}
\Der^\hdot_R(A, K)\ \underset{{}^\eqref{der_def}}\iso\ 
\Hom^\hdot_{\bimod{A}}(\Om^1_RA,\ K)\ \iso \ 
\Der^{\hdot-1}_A(T^\hdot_A(\Om^1_RA),\ K)\
=\ \Der^{\hdot-1}_A(\Om_RA,\ K).
\end{equation}

Let $\th\in \Der^\hdot_R(A,K)$ be a graded derivation.  Applying to
$\th$ the composite isomorphism above one obtains a graded derivation
$i_\th\in \Der^{\hdot-1}_A(\Om_RA,\ K)$, called {\em contraction with}
$\th$.

To define contraction on differential forms, we put $K=\Om^\hdot_RA$. 
For any $\th\in
\Der_RA$, the image of $\th$ in 
$\Der^{\hdot-1}_A(T^\hdot_A(\Om^1_RA),\ K)$ above produces a graded derivation,
which we also call $i_\th$,
on $\Om^\hdot_R A = T^\hdot_A(\Om^1_RA)$ (in fact, it is an $A$-derivation,
not merely an $R$-derivation).  
For a one-form $u\,\d v$, we have
 $i_\th(u\,\d v)=u\cdot\th(v)$.

Next, one  defines a  Lie derivative
map $L_\th:  \Om^\hdot_RA\to \Om^\hdot_RA$
by the Cartan formula
$L_\th:=i_\th\, \d +\d\, i_\th$.
Here, the expression on the right hand side is the super-commutator
of $i_\th$, a graded derivation of degree $-1$,
with $\d$, a graded derivation of degree $+1$.
It follows that $L_\th$ is a degree 0 derivation
of the algebra $\Om^\hdot_RA$.
For a one-form $u\,\d v$ we have
 $L_\th(u\,\d v)=(\th(u))\,\d v+ u\,\d(\th(v))$.

Each of the maps $L_\th$ and $i_\th$  descends to a
well-defined operation on the de Rham complex
$\DR^\hdot_R A=(\Om_RA^\hdot)_\cyc$.

\subsection{Lie derivative and contraction for double derivations}
\label{li_sec} 
Write $\dder_R\a:=\Der_R(A,A\o A)$ for the vector space of 
 $R$-derivations $A\to A\o A$, which we refer to as ``double derivations.'' When $R=\k$ we write $\dder \a$.
Double derivations do {\em not}
give rise to  natural operations on the DG algebra
$\Om_R^\hdot \a$ itself.
Instead, for
any  double derivation
 $\Th: A\to A\o A$, we will
define associated contraction and Lie derivative operations,
which are 
double derivations $\Om_R A\to \Om_R A\o \Om_R A$.

To do so, we apply the general construction based on \eqref{K}
in the special case where
$K=\Om_R^\hdot A\o \Om_R^\hdot A$,
an $\Om_R^\hdot A$-bimodule with respect to the outer action.
For $\Th\in\dder_R \a$, we consider the composition  
$A\to A\o A= \Om_R^0 A\o \Om_R^0 A\into \Om_R^\hdot A\o\Om_R^\hdot A=K$.
This is  a degree zero derivation $A\to K$ so, using \eqref{K}, we obtain 
a contraction  map $\Om_R A\to\Om_R A\o \Om_R A$.
This map, to be denoted $\bbi_\Th$, is an $A$-linear graded derivation of degree $(-1)$.

In the special case $n=1$, the contraction $\bbi_\Th$  reduces to the map
$
i_\Th:\ \Om_R^1\a\to A\o A,\ \al\mapsto (i_\Th'\al)\o
(i_\Th''\al),
$
that corresponds to the derivation  $\Th\in\dder_R \a$
via  the  canonical bijection \eqref{der_def}.
More generally,
for $n\geq 1$, separating individual homogeneous components, 
the contraction gives maps
$$\bbi_\Th:\ \Om_R^n A\lto \bplus_{1\leq k\leq n}\
\Om_R^{k-1} A\o\Om_R^{n-k} A.$$
Explicitly,  for any 
 $\al_1,\ldots,\al_n\in \Om_R^1\a$, one finds
\begin{equation}\label{iformula}
\bbi_\Th(\al_1\al_2\ldots\al_n)\ =\
\sum_{1\leq k \leq n} (-1)^{k-1}\cd
\al_1\ldots\al_{k-1}\,(i_\Th'\al_k)\o (i_\Th''\al_k)\,
\al_{k+1}\ldots\al_n. 
\end{equation}

Next, we define the Lie derivative.
To this end, one first extends the de Rham
differential on $\Om_R A$ to a degree one map
$\d:\
\Om_R^\hdot A\o \Om_R^\hdot A\to \Om_R^\hdot A\o \Om_R^\hdot A$
defined, for any $\al\in \Om_R^p A$ and $\be\in\Om_R^q A$,
by the formula $\d(\al\o\be):=(\d\al)\o\be +(-1)^p \al\o(\d\be).$
Now, given $\Th \in\dder_R A$, we use Cartan's formula as a definition
and put
$\BL_\Th:=\bbi_\Th\,\d+\d\,\bbi_\Th$. This gives a  graded double derivation
$\BL_\Th:\ \Om_R A\to \Om_R A\o \Om_R A$ of degree $0$.

\section{Extended Karoubi-de Rham complex} \label{extkadrcplx}
\subsection{Cyclic quotients for free products}\label{grading} 
Given a graded algebra $B=\bplus_q\ B^q$, we
equip  $B_t=B*\k[t]$, a free product  algebra,
with a {\em bi}grading 
$B_t=\bplus_{p,q}\ B_t^{p,q}$
such that the homogeneous component $B^q$, of the
subalgebra $B\sset B_t$ is assigned
bidegree $(0,q)$ and the variable $t$
 is assigned
bidegree $(2,0)$.
Thus, the  $p$-grading counts twice the number
of occurrences of the variable $t$.

The bigrading on $B_t$ descends to a bigrading
 $(B_t)_\cyc= \bplus_{p,q}\ (B_t)^{p,q}_\cyc$ on the
super-commutator quotient. In particular,
 $(B_t)^{0,\hdot}_\cyc=B^\hdot_\cyc$.
Further, it is clear that the assignment
$u_1\,t\,u_2\mto (-1)^{|u_1|\cdot |u_2|}u_2\,u_1$ yields an isomorphism
$(B_t)^{2,\hdot}_\cyc \cong B$.
More generally, for any $n\geq 1$, the
space $(B_t)^{2n,\hdot}_\cyc$
is spanned by {\em cyclic} words
$u_1\,t\,u_2\,t\ldots t\,u_n\,t$, where cyclic means that, for
instance, one has
$u_1\,t\,u_2\,t\,u_3\,t=(-1)^{|u_1|(|u_2|+|u_3|)}t\,u_3\,t\,u_1\,t\,u_2$ (modulo
commutators).

Given a graded vector space $V$, we let  the group $\Z/n \Z$ act on
 $V^{\o n}=T^n_\k V$ by cyclic permutations tensored by
the sign character and write 
\[
V^{\o n}_\cyc:= V^{\o n}/(\Z/n\Z).
\]

Thus, the assignments
$u_1\,t\,u_2\,t\ldots t\,u_n\mto$
$u_1\o u_2\o\ldots \o u_n$ and
$u_1\,t\,u_2\,t\ldots t\,u_n\,t\mto$
$u_1\o u_2\o\ldots \o u_n$
yield
natural   vector
space  isomorphisms 
\begin{equation}\label{triv}
B^{2n-2,\hdot}_t\iso (B^\hdot)^{\o n}\quad\text{and}\quad(B_t)^{2n,\hdot}_\cyc\cong
(B^\hdot)^{\o n}\cyclic,\quad\forall n=1,2,\ldots. 
\end{equation}

Any graded derivation $\th: B\to B$ has a natural extension to 
a  graded derivation  $\th_t: B_t\to B_t$ that restricts to 
$\th$ on the subalgebra $B\sset B_t$ and sends $t$ to $0$.
Further, any double derivation $\Th: B\to B\o B$
gives rise to a unique derivation
$\Th_t:  B_t\to B_t$ such that
$\Th_t(b)=\Th'(b)\,t\,\Th''(b)$, for any $b\in B$,
and $\Th_t(t)=0$.

Now fix an ungraded algebra $B$. We   view it as a graded
algebra concentrated in degree zero. Then, the second grading 
on $B_t$ becomes trivial, so the bigrading  effectively
reduces to a grading $B_t=\bplus_p B_t^{2p}$, by even integers. Then,
we immediately deduce the following.

\begin{lem}\label{f_tlem} $\Der^p_{\k[t]}(B_t,\,B_t)=0$
for any $p<0$. 
The assignments $\th\mto \th_t$ and $\Th\mto\Th_t$,
yield  isomorphisms
$\Der(B,B)\iso\Der^0_{\k[t]}(B_t,\,B_t)$ and
$\dder(B,B)\iso\Der^2_{\k[t]}(B_t,\,B_t)$, respectively.
\end{lem}

\subsection{The extended de Rham complex}\label{OMT}
We are going to introduce an enlargement of the
noncommutative de Rham complex $\Om A$.
This  enlargement is a  DG algebra that has three equivalent
definitions according to the following lemma.

\begin{lem} There are  natural   algebra  isomorphisms
\begin{equation}\label{psi}
\Om_{\k[t]}(A_t)\,\cong\,
  (\Om A)*\k[t]\,\cong\, T_A(A^{\o 2}\oplus\oma).
\end{equation}

The differential on $T_A(A^{\o 2}\oplus\oma)$
obtained by transporting the de Rham differential on
$\Om_{\k[t]}(A_t)$ via the isomorphisms in \eqref{psi}
is the derivation induced by the composite
$i_\De\ccirc\d:\ A \to \Om^1A\into A\o A$ (using \eqref{fund}).
\end{lem}
\begin{proof} The algebra  $\Om_R(A*R)$ is the quotient of the algebra
$\Om(A*R)$ by the two-sided ideal generated by 
the space $\d R\sset \Om^1R\sset\Om^1(A*R)$.
Since  $\Om(A*R)\cong (\Om A)*(\Om R)$, by
\eqref{AB}, we deduce  a  DG algebra isomorphism
\begin{equation}\label{om*}\Om_R(A*R)\cong (\Om A)*R.
\end{equation}

Applying this in the special case 
$R=\k[t]$ yields the first isomorphism in \eqref{psi}.

To prove the second isomorphism in \eqref{psi}, fix  an $A$-bimodule $M$.
View
 ${A^{\o 2}}\oplus M$ as an $A$-bimodule. The assignment 
$(a'\o a'')\oplus m\mto a'ta''+m$ clearly gives an
$A$-bimodule map ${A^{\o 2}}\oplus M\to
(T_AM)_t$. This map can be extended, by the universal
property of the tensor algebra,
to an algebra morphism $T_A({A^{\o 2}}\oplus M)\to
(T_AM)_t$.
 To show that this morphism is an isomorphism,
we explicitly construct an inverse map as follows.

We start with a natural algebra embedding
$\dis f: T_AM\into T_A({A^{\o 2}}\oplus M)$,
induced by the $A$-bimodule embedding
$\dis M=0\oplus M\into{A^{\o 2}}\oplus M$.
Then, by the universal property of free products,
we can (uniquely) extend the map $f$ to an
algebra  homomorphism $(T_AM)_t=(T_AM)*\k[t]\to T_A({A^{\o 2}}\oplus M)$
by sending $t\, \mto \, 1_A\o 1_A\in A^{\o 2}\sset
T^1_A({A^{\o 2}}\oplus M)$.
It is straightforward to check that the resulting
homomorphism is indeed an inverse
of the homomorphism  constructed earlier.

Applying the above in the special case $M=\oma$ 
yields the second isomorphism  of the lemma.

The statement of the lemma concerning differentials is left
to the reader.
\end{proof}

 It is convenient to introduce a special notation
 $\Om_t A:=\Om_{\k[t]}(A_t)$. This algebra comes equipped with  a natural 
{\em bi}-grading
$\Om_t A=\oplus_{p,q\geq 0}\, \Om^{2p,q}_t A$, where
the even $p$-grading is induced
from the one on $A_t$, and the $q$-component
corresponds to the grading induced by the
natural one on $\Om^\hdot A$. It is easy to see that
the  $p$-grading corresponds,
under the isomorphism \eqref{psi}
to the grading on $(\Om A)*\k[t]$ that
counts twice the number
of occurrences of the variable $t$.

The {\em extended
de Rham complex} of $A$ is defined as a super-commutator quotient
$$
\DR_t A:=\DR_{\k[t]}(A_t)=(\Om_{\k[t]}(A_t))_\cyc\cong
((\Om A)_t)_\cyc.
$$
The bigrading on $\Om_t A$ clearly descends to
a  bigrading on the
extended de Rham complex of $A$.
The de Rham  differential has  bidegree $(0,1)$:
$$\DR_t A=\oplus_{p,q}\, \DR^{2p,q}_t A,\qquad
\d:\
 \DR_t^{2p,q} A\to\DR^{2p,q+1}_t A.
$$

Next, we
use the identification \eqref{triv} for $B:=\Om A$, and
 equip $(\Om^\hdot\a)^{\o p}$ with the tensor product grading
that counts the total degree of differential
forms involved, e.g., given $\al_i\in \Om^{k_i}A,\ i=1,\ldots,p$,
for $\al:=\al_1\o\ldots\o\al_p\in (\Om^{\,}\a)^{\o p}\cyclic$,
we put $\deg\al:=k_1+\ldots+k_p$. Then, we get
\begin{equation}\label{alternate}
\DR_t^{2p,q} A=
\begin{cases} 
\DR^q A &\text{if}\en p=0;\\
\text{degree $q$ component of}\en(\Om^\hdot\a)^{\o p}_\cyc
 &\text{if}\en p\geq 1.
\end{cases}
\end{equation}

\subsection{Operations on the extended de Rham complex}\label{dergr}
Let $\th\in\Der A$.
The Lie derivative  $L_\th$ is a
 derivation of the algebra $\Om^\hdot A$.
Hence, there is an associated
  derivation $(L_\th)_t:
(\Om^{\,}A)_t\to (\Om^{\,}A)_t$; see Lemma \ref{f_tlem}(ii).
On the other hand, we may first extend $\th$ to a derivation
$\th_t: A_t\to A_t$, and then consider the Lie derivative
$L_{\th_t}$, a derivation of the algebra
$\Om_{\k[t]}(A_t)=\Om_tA$,
 of bidegree
$(0,0)$. Very similarly, we also have
 graded derivations, $(i_\th)_t$ and
$i_{\th_t}$, of bidegree $(0,-1)$.

It is immediate to see that the 
two procedures above agree with each other in the
sense that, under the identification
$\Om_{\k[t]}(A_t)\cong(\Om^{\,}A)_t$
provided by \eqref{psi}, 
\begin{equation}\label{twoconstr}
L_{\th_t}=(L_\th)_t,\quad\text{and}\quad
i_{\th_t}=(i_\th)_t.
\end{equation}

Next, we consider operations induced by double derivations.
\begin{prop}\label{li} Any double derivation $\Th\in\dder A$
gives   a
canonical Lie derivative  operation, 
a graded derivation $L_\Th\in \Der_{\k[t]}(\Om_tA)$
of bidegree $(2,0)$, and a  contraction
 operation,  a graded 
derivation $i_\Th\in \Der_{\k[t]}(\Om_tA)$
of bidegree $(2,-1)$.
\end{prop}

\begin{proof} Given  $\Th\in \dder{A}$,  we  first extend it
to a free product derivation $\Th_t: A_t\to A_t$,
as in \S\ref{grading}. Hence, there are associated
  Lie derivative $L_{\Th_t}$, and contraction
$i_{\Th_t}$,  operations  on the complex $\Om_{\k[t]}(A_t)$,
of  relative differential forms on the algebra $A_t$.
Thus, we may use \eqref{triv} to interpret 
 $L_{\Th_t}$ and
$i_{\Th_t}$ as operations on  $\Om_tA$, to be denoted by  $L_\Th$ and $i_\Th$,
respectively. \end{proof}

\begin{rem}\label{twodefs}
One can use the above defined actions of  double derivations on $\Om_tA$ to obtain
an alternative construction of the operations  introduced in \S\ref{li_sec}.
Specifically, we observe that  the operations $L_{\Th_t}$ and
$i_{\Th_t}$
may be viewed, thanks to the first isomorphism in \eqref{psi},
as  derivations of the algebra  $(\Om A)_t$ of degree  $2$ with respect to the
$p$-grading, that counts
 (twice) the number of occurrences of $t$.
Hence, applying Lemma \ref{f_tlem}, we conclude that
there exists a  unique double
derivation $\BL_\Th:  \Om^{\,}\a\to \Om^{\,}\a\o \Om^{\,}\a$
such that,
for the corresponding map $(\Om^{\,}\a)_t\to(\Om^{\,}\a)_t$, we have
$\dis 
L_{\Th_t}=(\BL_\Th)_t$. A similar argument
yields a double derivation $\bbi_\Th:  \Om^{\,}\a\to \Om^{\,}\a\o
\Om^{\,}\a$
such that $i_{\Th_t}=(\bbi_\Th)_t$. It is easy to check that
the derivations $\BL_\Th$ and $\bbi_\Th$ thus defined 
are the same as those introduced  in \S\ref{li_sec}
in a different way.
\erem

Both the Lie derivative and contraction operations on $\Om_tA$ 
descend to the commutator quotient. This way, we obtain the Lie derivative
 $L_\Th$ and the contraction
$i_\Th$  on $\DR_t A$, the extended de Rham complex.
Explicitly, using isomorphisms \eqref{alternate},
we can write the Lie derivative
 $L_\Th$ and  contraction
$i_\Th$ operations
 as 
chains of maps 
of the form
\begin{equation}\label{chainDR}
\DR A\too \Om^{\,}\a\too   
(\Om^{\,}\a)^{\o 2}\cyclic \too 
(\Om^{\,}\a)^{\o 3}\cyclic \too \ldots.
\end{equation}

 There are  some standard identities involving the Lie derivative 
 and contraction operations on $\Om A$ associated with ordinary derivations. Similarly,
the Lie derivative 
 and contraction operations on $\Om_tA$ resulting from Proposition
 \ref{li} satisfy the following identities:
\begin{equation}\label{relations}
L_\Th=\d\ccirc i_\Th+i_\Th\ccirc \d,
\quad
i_\Th\ccirc i_\Phi+ i_\Phi\ccirc i_\Th=0,
\quad i_\xi\ccirc i_\Th+i_\Th\ccirc i_\xi=0,
\quad\forall\Th, \Phi\in\dder\a,\ \xi\in\Der A.
\end{equation}
It follows, in particular, that the Lie derivative $L_\Th$
commutes with the de Rham differential $\d$.

To prove \eqref{relations}, one first verifies these identities
on the generators of the algebra $\Om_t A=(\Om^{\,}\a)_t$,
that is, on differential forms of degrees $0$ and $1$,
which is a simple computation.
The general case then follows by observing
that any commutation relation between (graded)-derivations
that holds on generators of the algebra holds true for
all elements of the algebra.

It is immediate that the induced operations 
on $\DR_t A$ also satisfy  \eqref{relations}.

\subsection{Reduced Lie derivative and contraction}\label{red_sec}
The second component
(the $q$-component) of
the 
bigrading on  $\DR_tA$ induces
a grading
on each of the  spaces $\dis(\Om^{\,}\a)^{\o k}\cyclic,\,k=1,2,\ldots$,
appearing in \eqref{chainDR}.
Each of the maps in \eqref{chainDR}
corresponding to  the  Lie derivative
induced by a double derivation
 $\Th\in\dder A$
preserves the   $q$-grading.
In the case of  contraction with $\Th$,
all maps in the corresponding chain  \eqref{chainDR}
decrease  the  $q$-grading by one.

The  leftmost map in \eqref{chainDR}, to be denoted $\bi_\Th$
in the contraction case and
$\bl_\Th$ in the Lie derivative case,
will be especially important for
us. 
These maps,  which we will call the  {\em reduced contraction}
and {\em reduced Lie derivative}, respectively, have the form
\begin{equation}\label{i_explicit}
\bi_\Th: \ \DR^\hdot A\map \Om^{\hdot-1}A,
\quad\text{and}\quad
\bl_\Th: \
\DR^\hdot A\map \omb.
\end{equation}
Explicitly, we see from \eqref{iformula} that the operation $\bi_\Th$,
for instance,  is given,
for any $\al_1,\al_2,\ldots,\al_n\in\Om^1 A$, 
by the following  formula:
\begin{align}\label{bi_formula}
\bi_\Th(\al_1&\al_2\ldots\al_n)=
\sum_{k=1}^n
(-1)^{(k-1)(n-1)}\cd  (i_\Th''\al_k)\cd
\al_{k+1}\ldots\al_n\,\al_1\ldots\al_{k-1}\cd (i_\Th'\al_k).
\end{align}

An {\em ad hoc}  definition of the maps in \eqref{i_explicit} via
 explicit formulas like \eqref{bi_formula} was first given in \cite{CBEG}.
Proving the commutation relations
 \eqref{relations} using  explicit formulas is, however, very
 painful; 
this was carried out
in  \cite{CBEG} by rather long brute-force computations.
Our present approach based on the free product construction
yields the  commutation relations  for free.

\subsection{The derivation $\Delta$}\label{vdb}
There is a distinguished double
derivation 
$$\Delta: A\to A\otimes A,\quad
a\mapsto 1\o a-a\o 1.
$$

The corresponding contraction map
$i_\Delta: \oma\to  A\otimes A$ is the tautological
embedding \eqref{fund}. Furthermore, the derivation
$\Delta_t: A_t\to A_t$ associated with $\Delta$, cf
\S\ref{grading},
equals $\ad t: u\mapsto t\cdot u-u\cdot t$.
Hence, the Lie derivative map
$L_\Delta: \Om_tA\to\Om_tA$ reads
$\om\mapsto \ad t(\om):=t\cdot \om-\om\cdot t$.

\begin{lem}\label{biDelta} \vi For any $a_0,a_1,\ldots, a_n\in A$, 
\begin{align*}
\bi_\Delta(a_0\,\d a_1\ldots \,\d a_n)
=\sum_{1\leq k \leq n}(-1)^{(k-1)(n-1)+1}\cd
[a_k,\ \d a_{k+1}\ldots \d a_n\,a_0\,\d a_1\ldots \d a_{k-1}].
\end{align*}

 \vii In  $\Om^\hdot A$, we have 
$\dis
\bi_\Delta\ccirc \d+\d\ccirc \bi_\Delta=0$
and $\d^2=(\bi_\Delta)^2=0;$ similar equations
also hold in
 $\DR_t^\hdot A$, with $i_\Delta$ in place of
$\bi_\Delta$.
\end{lem}
\begin{proof} Part (i) is verified
by a straightforward computation based on 
formula \eqref{iformula}.
We claim next that, in $\Om_tA$, we have $L_\De=\ad t$.
Indeed, it suffices to check this equality  on the generators
of the algebra $\Om_t A$. It is clear that
$L_\De(t)=0=\ad t(t)$, and it is easy to see
that both derivations agree on zero-forms and on one-forms.
This proves the claim. 

The equation $\bi_\Delta\ccirc \d+\d\ccirc \bi_\Delta=0$,
of part (ii) of the lemma, now follows
 by the Cartan formula on the left of \eqref{relations}, since
the equation  $L_\De=\ad t$ clearly implies that
the map $L_\De: \DR_t A\to\DR_t A$, as well as the map  $\bl_\Delta$,
vanishes. Finally, the formula of part (i) shows that
the image of the map $\bi_\Delta$ is contained in $[A, \Om A]$.
Hence, we deduce $(\bi_\Delta)^2(\Om A)\sset
\bi_\Delta([A, \Om A])=0$, since the map $\bi_\Delta$ vanishes on
commutators.
\end{proof}

 Let $A_\tau:=A*\k[\tau]$ 
be the graded algebra
such that $A$ is placed in degree zero and  $\tau$ is an {\em odd}
variable placed in degree $1$. Let $\frac{d}{d\tau}$ be the degree $-1$
derivation of the algebra $A_\tau$ that annihilates $A$ and
satisfies $\frac{d}{d\tau}(\tau)=1$. Similarly, let 
$\tau^2\frac{d}{d\tau}$  be the degree $+1$ graded
derivation of the algebra $A_\tau$ that annihilates $A$ and satisfies
$\tau^2\frac{d}{d\tau}(\tau)=\tau^2$. For any homogeneous element
$x\in A_\tau$, put $\ad\tau(x):=\tau x-(-1)^{|x|}\tau;$ in particular,
one finds that $\ad\tau(\tau)=2\tau^2$.

It is easy to check that each of the derivations
$\frac{d}{d\tau},\
\tau^2\frac{d}{d\tau}$, and $\ad\tau-\tau^2\frac{d}{d\tau}$ squares to
zero.
\begin{claim} \vi {\em The following assignment gives a graded algebra embedding:
$$
j:\ \Om_t A\ \into\ A*\k[\tau],\quad t\mto \tau^2,\quad a_0\,\d a_1\ldots\d a_n\mto
a_0\cd[\tau,a_1]\cd\ldots\cd[\tau,a_n],
$$
Moreover, the above map  intertwines the contraction operation ${\mathbf
i}_\Delta$ with the differential $\tau^2\frac{d}{d\tau}$, and
 the Karoubi-de Rham differential $\d$ with  
the differential $\ad\tau-\mbox{$\tau^2\frac{d}{d\tau}$}$.

\vii The image of the map $j$ is annihilated by the derivation
$\frac{d}{d\tau}$.

\viii The complex $\big((A_\tau)_\cyc,\ \frac{d}{d\tau}\big)$ computes
cyclic homology of the algebra $A$.}\qed
\end{claim}
\noindent
We will neither use nor prove this result; cf.~\cite[Proposition
1.4]{CQ1} and \cite[\S 4.1 and Lemma ~4.2.1]{KS}.

\section{Applications to Hochschild and  cyclic homology}
\subsection{Hochschild homology}\label{hoch_sect}
Given an algebra $A$ and an $A$-bimodule $M$, we
let $H_k(A,M)$ denote the $k$-th  Hochschild
homology group of $A$ with coefficients in $M$.
Also, write $[A,M]\sset M$ for the
$\k$-linear span of the set $\{am-ma\mid a\in A,m\in M\}$.
Thus, 
$H_0(A,M)= M/[A,M]$.

We extend some ideas of Cuntz and Quillen  \cite{CQ2} to obtain our first important result.

\begin{thm}\label{hoch_thm} For any unital $\k$-algebra $A$, there
is a natural graded space isomorphism
$$H_\idot(A,A)\cong\Ker[\bi_\De: \DR^\hdot A\to \Om^{\hdot-1}A].$$
\end{thm}

To put  Theorem \ref{hoch_thm} in context, recall that
Cuntz and Quillen used 
noncommutative differential forms 
to compute  Hochschild
homology. Specifically,  following
 \cite{CQ1} and \cite{CQ2}, consider a complex
$\ldots\stackrel{\b}\map\Om^{2}A
\stackrel{\b}\map\Om^{1}A\stackrel{\b}\map\Om^{0}A\map 0$.
Here, $\b$ is the {\em Hochschild differential}  given by the
formula
\begin{equation}\label{b} \b:\ \al\,\d a\lto(-1)^{n }\cd[\al, a],\qquad\forall
 a\in A/\k,\ \al\in\Om^{n }A,\,
n >0.
\end{equation}

It was shown in \cite{CQ2} that the complex $(\Om^\hdot A, \b)$
can be identified with the standard  Hochschild chain
complex. It follows that
$H_\idot(\Om^{\,}\a,\b)=H_\idot(A,A)$ 
are the  Hochschild
homology groups of $A$.

 Theorem \ref{hoch_thm} will follow directly from Proposition
\ref{bkappa} below (see the discussion after this proposition).
Proposition \ref{bkappa} itself will be proved in \S \ref{pf1}.

\begin{rem} A somewhat more geometric interpretation of Theorem
  \ref{hoch_thm}, from the point of view of representation functors,
  is provided by the map \eqref{first}: see Theorem \ref{rep_thm} of
  \S \ref{repf} below.
\end{rem}

\subsection{An application} The algebra  $A$ is said to be {\em
connected} if the following
sequence
is exact:
\begin{equation}\label{exact}
0\map \k\map \DR^0 A\stackrel{\d}\map\DR^1 A.
\end{equation}

\begin{prop}\label{ham_lemma}
Let  $A$ be a {\em connected}  algebra such that $H_2(A,A)=0$. 
Then,   \smallskip

\npb{$H_1(A,A)=(\DR^1 A)\closed$ and $(\DR^2 A)\exact=(\DR^2 A)\closed$.}

\npb{There is a natural  vector space
isomorphism $(\DR^2 A)\closed\iso[A,A]$.}
\end{prop}

\begin{proof} We will freely use the notation of 
\cite[\S4.1]{CBEG}.
According to \cite[Proposition 4.1.4]{CBEG},  for any connected algebra
$A$, there is a  map
  ${\widetilde{\munc}}$,  a lift of the
{\em noncommutative moment map}, that fits into
the following  commutative diagram:
\begin{equation}\label{sqdiag}
\xymatrix{
\DR^1 A\ar[rr]^<>(0.5){\d}\ar@{->>}[d]_<>(0.5){\bi_\Delta}&&
(\DR^2 A)\closed
\ar@{.>}[dll]_<>(0.5){\widetilde{\munc}}\ar[d]_<>(0.5){\bi_\Delta}\\
{[A,A]}^{^{}}\ar[rr]^<>(0.5){\d}&&[A,\Om^1 A].
}
\end{equation}

Since
$H_2(A,A)=0$, we deduce from
 the short exact sequence of
 Theorem \ref{hoch_thm} for $n=2$
 that the right vertical map
$\bi_\Delta$ in diagram \eqref{sqdiag} is injective.
It follows, by commutativity  of the lower right
triangle in  \eqref{sqdiag},
that the map  ${\widetilde{\munc}}$ must be injective.
Thus, from Theorem \ref{hoch_thm} for $n=1$, we deduce
$H_1(A,A)=\ker(\bi_\Delta: \DR^1 A \to [A,A])=(\DR^1 A)\closed$.

Next,  the left vertical map $\bi_\Delta$ in the diagram
is given by the formula $a\,\d b\mto [a,b]$. Hence, this map
is surjective. Since the map   ${\widetilde{\munc}}$ 
is injective, it follows
from the commutativity  of the upper left
triangle in  \eqref{sqdiag} that 
 ${\widetilde{\munc}}$, as well as the map $\d$ in the
top row of diagram  \eqref{sqdiag}, must  be  surjective.
We conclude that  
the map  ${\widetilde{\munc}}$
 yields an
isomorphism $(\DR^2 A)\closed\iso [A,A]$
and also that $(\DR^2 A)\exact=(\DR^2 A)\closed$.
\end{proof}

Proposition \ref{ham_lemma} can be easily extended to  a relative
setting 
where 
the algebra $A$ contains
a  subalgebra of the form $R =\k^r$, for some $r \geq 1$.
Then, the correct  relative counterpart of the commutator space $[A,A]$
turns out to be the  subspace $[A,A]^R\sset [A,A]$, 
 formed
by the elements  which commute with $R$.
The corresponding formalism has been worked out in \cite{CBEG}.
The relative version of Proposition \ref{ham_lemma} reads as follows.

\begin{cor}\label{mucor} Let  $R=\k^r$.
Let $A$ be an algebra containing $R$ and such that
the sequence $0\to R \to A\to \DR_R^1A$ is exact and
$H_2(A,A)=0$.
Then,
there is a natural  vector space
isomorphism $(\DR^2_RA)\closed\iso[A,A]^R$.\qed
\end{cor}

An important example where the above corollary applies is the case
where $A$ is the path algebra of a quiver with $r$ vertices.

\begin{rem} The isomorphism of
Proposition  \ref{ham_lemma} and Corollary \ref{mucor} plays a role
in the theory of Calabi-Yau algebras; see \cite[Claim 3.9.11]{Gi2}.
\end{rem}

\subsection{Cyclic homology} We recall some standard definitions, following
\cite[Chapter 2 and p.~162]{L}. For any graded vector space
$M=\oplus_{i\geq 0}\,M^i$, we introduce a $\Z$-graded
$\k[t,t^{-1}]$-module
\begin{equation} \label{bomteq}
M \hat\o\k[t,t\inv]\ :=\ \underset{n \in \Z}\bplus\
\left(\prod\nolimits_{i \in \Z}\
 t^{i}\, M^{n-2i}\right),
\end{equation}
where the grading is such that the space
$M\sset M \hat\o\k[t,t\inv]$  has the natural grading,
and $|t|:=2$.

Below, we will use a complex of {\em reduced} differential forms,
defined by setting 
$\bom^{_0}:=\Om^0A/\k=A/\k$ and
$\bom^{_k}:=\Om^kA$ for all $k>0$.
Let $\bom^{_\bullet} :=\bigoplus_{k\geq 0}\,\bom^{_k}$.
The Hochschild differential induces a $\k[t,t^{-1}]$-linear
differential $\b: \bomt \to \bomt$ of degree $-1$.

We also have the Connes differential
$\B: \bom^{_\bullet} \to\bom^{_{\hdot+1}}$ \cite{Co}. Following
Loday and Quillen \cite{LQ},
we  extend it to a $\k[t,t^{-1}]$-linear 
differential on $\bomt$ of degree $+1$.
It is known that $\B^2=\b^2=0$ and 
$\B\ccirc \b+\b\ccirc \B=0$. Thus, the map
$\B+t\cdot\b: \bomt\to\bomt$
gives a degree $+1$ differential on $\bomt$.

Write  $HP_{-\idot}(A)$, where, `$-\idot$' denotes
\textit{inverting} the degrees, for the  {\em reduced periodic cyclic
homology} of $A$
as defined in \cite{LQ} or  \cite[\S 5.1]{L},  using  a complex with
differential of degree $-1$.
According to \cite{CQ2}, the groups $HP_{-\idot}(A)$
turn out to be isomorphic to homology groups of the complex
$(\bomt, \B+t\cdot\b)$, with  differential of degree 
$+1$ (which is why we must invert the degrees).
It is known that  the action of multiplication by $t$
yields periodicity isomorphisms $HP_\idot(A)\cong HP_{\idot+2}(A).$
 Thus, up to isomorphism, there are only two groups,
$HP_{\text{even}}(A) := HP_0(A)$, and $HP_{\text{odd}}(A) := HP_1(A)$.

Next, we compose  the map 
 $\bi_\Delta: \DR^\hdot A\to\Om^{\hdot-1}A$
with the natural projection $\omb\onto \DR^\hdot A$
to obtain a map  $\omb\to\Om^{\hdot-1}A$.
The latter map descends to a map
 $\bom^{_\bullet} \to\bom^{_{\bullet-1}}$.
Furthermore, we may extend this last map, as well as the de
Rham differential
$\d: \bom^{_\bullet} \to\bom^{_{\hdot+1}}$,
to  $\k[t,t^{-1}]$-linear maps
 $\bomt\to\bomt$, of degrees $-1$
and $+1$, respectively.

The resulting maps $\d$ and $\bi_\Delta$ satisfy
$\d^2=(\bi_\Delta)^2=0$ and
$\d\ccirc\bi_\Delta+\bi_\Delta\ccirc d=0$, by Lemma \ref{biDelta}(ii).
Thus, the map $\d+t\cdot\bi_\Delta$
gives a degree $+1$ differential on $\bomt$.
This differential may be thought of as
some sort of equivariant differential for the `vector field'
$\Delta$.

The following theorem, to be proved
 in \S\ref{pf2} below, is one of the main results of the paper. It
shows the importance of the reduced contraction
map $\bi_\Delta$ for cyclic homology.

\begin{thm}\label{main} 
The homology of the  complex
$(\bomt, \d+t\cdot\bi_\Delta)$ 
is isomorphic to
$HP_{-\idot}(A)$, the reduced periodic cyclic homology of $A$ (with inverted degrees).
\end{thm}

\begin{rem}[Hodge filtration]
In \cite[\S 1.17]{K08}, Kontsevich considers
a `Hodge filtration' on periodic cyclic homology. In terms of Theorem \ref{main},
the Hodge filtration ${\mathsf F}^\hdot_{_{\text{Hodge}}}$ may be defined as follows:
\vskip 2pt

\npb{${\mathsf F}^n_{_{\text{Hodge}}} HP_{\text{even}}\ $ consists of those
classes representable by sums
$\
\sum_{i \geq n}\ t^{-i} \gamma_{2i}, \quad \gamma_{2i} \in \bom^{2i}
$;}
\vskip 1pt

\npb{${\mathsf F}^{n+\frac{1}{2}}_{_{\text{Hodge}}} HP_{\text{odd}}\ $ consists of those classes
representable by sums
$\
\sum_{i \geq n} t^{-i} \gamma_{2i+1}, \quad \gamma_{2i+1} \in \bom^{2i+1}.
$}
\end{rem}


\subsection{Gauss-Manin connection}\label{gmsec} It is well-known that,
given a smooth family $p: \scr X \to S$ of complex schemes over a
smooth base $S$, there is a canonical flat connection on the relative
algebraic de Rham cohomology groups $H_{DR}^\hdot({\scr X}/S)$, called
the \textit{Gauss-Manin connection}.  More algebraically, let $A$ be a
commutative $\k$-algebra which is smooth over a regular subalgebra
$B\sset A$.  In such a case, the relative algebraic de Rham cohomology
may be identified with $HP^B_\idot(A)$, the relative periodic cyclic
homology; see, e.g., \cite{FT}.  The Gauss-Manin connection therefore
provides a flat connection on the relative periodic cyclic homology.

In \cite{Getz}, Getzler extended the definition of
 the Gauss-Manin connection to  a noncommutative setting.
Specifically,   let $A$ be a (not necessarily commutative) associative
algebra equipped with a {\em central} algebra embedding
$B=\k[x_1,\ldots,x_n]\into A$.
Assuming that $A$ is free as a $B$-module,
Getzler has defined a flat connection on $HP_\idot^B(A)$.
Unfortunately, Getzler's definition of the connection involves
quite complicated calculations in the Hochschild complex that
make it difficult to relate  his definition with
the classical geometric construction of  the Gauss-Manin connection
on de Rham cohomology.
Alternative approaches to the definition of Getzler's connection,
also based on homological algebra, were suggested more
recently by Kaledin \cite{K} and by
Tsygan \cite{T2}.

Below, we propose  a new, geometrically more transparent (we believe)
 approach
for the Gauss-Manin connection using the construction of
cyclic homology from the previous subsection.  Unlike earlier
constructions, our formula for the connection on  periodic cyclic homology
is identical, essentially,
 to the classic formula for the  Gauss-Manin connection
in  de Rham cohomology, though the objects involved
 have  different meanings. 

Our version of Getzler's
result reads as follows.

\begin{thm} \label{gmthm} Let $B$ be a commutative algebra.
Let $A$ be an associative algebra equipped
with a central algebra embedding $B\into A$ such that the quotient 
 $A/B$ is  a free $B$-module.
  
Then, there is a canonical
flat connection $\nabla_{GM}$ on $HP_{\idot}^B(A)$.
\end{thm}

\begin{notation} \vi Given an algebra $R$ and a subset $J\sset R$, let $\ldp J\rdp$
denote the two-sided ideal in $R$ generated by the set $J$.  

\vii For a
commutative algebra $B$,
we set
$\Om^\hdot\com B:=\La^\hdot_B(\Om\com^1B)$,  the  super-commutative DG algebra
of  differential forms, generated by the $B$-module
$\Om\com^1B$ of K\"ahler differentials.
\end{notation}

\noindent
\textit{Construction of the Gauss-Manin connection.} Given a
{\em central} algebra embedding $B\into A$, we
  may realize the relative periodic cyclic homology of $A$ over $B$ as
follows.  First, we define the following quotient DG algebras of
$(\Omega^\hdot A, \d)$:
\begin{gather*}
\oba  := \Omega^\hdot A / \ldp [\Omega^\hdot A,
\Omega^\hdot B] \rdp, 
\qquad 
\Om(A;B) := \oba  / \ldp \d B \rdp.
\end{gather*}

Thus, we have a super-central
DG algebra embedding $\Om^\hdot\com B\into\oba$
induced by 
the natural embedding $\Om^\hdot B\into \Om^\hdot A$.
We introduce the descending filtration
$F^\hdot(\oba )$ by powers of the ideal
$\ldp  \d B \rdp$. For the corresponding associated graded algebra,
there is a natural surjection
\begin{equation}\label{grass}
\Om^\hdot(A;B) \o_B \Omega^i\com  B \onto \gr^i_F \oba, \quad
\alpha \o
\beta \mapsto \alpha \beta,\quad \forall \alpha \in \oba ,\
 \beta \in \Omega^i\com  B.
\end{equation}

Below, we will also make use of the objects
 $\boba$ and  $\bom(A;B)$,  obtained  by killing
$\k\sset A = \Omega^0 A$.
Thus, 
$\bom(A;B)\ttc$ and $\boba 
\ttc$
are modules  over $\k[t,t^{-1}]$.
There is a natural descending filtration $F^\hdot$ on $\boba \ttc$
induced by
$F^\hdot(\oba )$ and such that $\k[ t, t^{-1} ]$
is placed in filtration degree zero. This filtration  is  obviously
stable under the differential
 $\d$. It is also stable under  the differential $t\cdot\bi_\Delta$
since the commutators that appear in $t\cdot\bi_\Delta(\omega)$ 
(see Lemma \ref{biDelta}(i)) vanish, by
definition of $\oba $. Therefore, the map \eqref{grass}
induces  a morphism of double complexes,
equipped with the differentials $d \o_B \Id$
and $t\cdot\bi_\Delta \o_B \Id$,
\begin{equation}\label{gras}
\bom^{_\hdot}(A;B)\ttc\
 \o_B\
 \Omega^i\com  B\ 
 \onto\
\gr^i_F \boba \ttc.
\end{equation}

We will show in \S\ref{pfgm} below that the assumptions of
Theorem \ref{gmthm} ensure that
the map \eqref{grass} is an isomorphism. 

Assume this for the moment and consider the standard spectral sequence 
associated with the filtration $F^\hdot(\boba \ttc)$.
The first page of this sequence consists of terms
$\gr_F(\boba \ttc)$. 
 Under the above
assumption, the LHS of \eqref{gras}, summed over all $i$, composes
the first page of the spectral sequence of
$\Big(F^\hdot(\boba \ttc),\ \d + t\cdot\bi_\Delta\Big)$.
  Then, for the second page of the
spectral sequence we get
$$E^2=H^\hdot\big(\bom(A;B) \ttc,\
\d + t\cdot\bi_\Delta\big)\
\o_B\ \Om^\hdot\com B.$$

We now describe the differential $\nabla$  on the second page.
 Let 
\begin{equation}\label{nab}\nabla_{GM}:\  H^\hdot(\bom(A;B) \ttc) \
\too\ H^\hdot\big(\bom(A;B) \ttc\big)\ \o_B\
 \Omega^1\com  B
\end{equation}
be the restriction of  $\nabla$ to degree zero.  Then we immediately
see that 
$$
\nabla(\alpha \o \beta) = \nabla_{GM}(\alpha) \wedge \beta + (-1)^{|\alpha|} \alpha \o
 (\ddr   \beta),
\qquad
\nabla_{GM}(b \alpha) = b\nabla_{GM}(\alpha) + (-1)^{|\alpha|} \alpha \o (\ddr b),\qquad \forall b
 \in B,
$$
where now $ \ddr  $ is the usual de Rham differential. 
From these equations, we deduce that
the map $\nabla_{GM}$, from \eqref{nab},
gives a flat connection on $H^i\big(\bom(A;B)\ttc\big) $ for all
 $i$.  

 Explicitly, we may describe the connection $\nabla_{GM}$ as follows.
 Suppose that $\bar\alpha \in \bom(A; B)$ has the property that
 $(\d + t\cdot\bi_\Delta)(\bar\alpha) = 0$.  Let $\alpha \in \bom^B A$ be any
 lift, and consider $(\d + t\cdot\bi_\Delta)(\alpha)$. This must lie in
 $\ldp \d B \rdp$, and its image in $\Om(A;B)\o_B \Omega^1\com B$ is
 the desired element.\hfill$\lozenge$

\begin{rem}\label{getzrem}
 In  \cite{Getz}, Getzler  takes $B = \k\lll x_1, \ldots, x_n \rrr$, and takes
 $A$ to be a {\em formal} deformation over $B$ of an associative algebra $A_0$.
Although such a setting is not formally covered by
Theorem \ref{gmthm}, our construction of the Gauss-Manin connection
still applies. 

To explain this, 
 write $\fr\sset B= \k\lll x_1, \ldots, x_n \rrr$  for the
augmentation ideal of the formal power series without constant term.
Let  $A_0$ be a $\k$-vector space with a fixed nonzero element $1_A$,
and let $A=A_0\lll x_1, \ldots, x_n \rrr$
 be the
$B$-module of formal power series with coefficients in $A_0$.
We equip $B$ and $A$ with the $\fr$-adic topology, and
view $B$ as a $B$-submodule in $A$ via the embedding $b\mapsto b\cdot 1_A$.

\begin{cor}\label{getzex} Let $\star: A\times A \to A$ be a $B$-bilinear, continuous
associative  (not necessarily commutative)
product that makes $1_A$ the unit element.
Then, the conclusion of Theorem \ref{gmthm} holds
for $HP^B_\idot(A)$.
\end{cor}
\end{rem}

\section{Proofs}\label{pf_main}

\subsection{The Karoubi operator}\label{hoh} 
Throughout this section,
we fix  an algebra $A$ and abbreviate 
$\Om^n:={\Om^n  A} $ and $\Om:=\oplus_n\,\Om^n$.

Given an $A$-bimodule $M$, 
 put
$M\br:= M/[A,M]=H_0(A,M)$. 
In particular, an algebra homomorphism $A\to B$ makes
$B$  an $A$-bimodule. In such a case, one has
a canonical projection $B\br=B/[A,B]\onto B_\cyc=B/[B,B]$.
This applies  for $B=\omh$. So, we get
a natural projection $\omh \br\to\DR^\hdot A$,
which is not 
an isomorphism, in general.

Following Cuntz and Quillen \cite{CQ2},
we consider a diagram
$$
\xymatrix{
{{\Om^{0}}\;} \ar@<1ex>[r]^<>(0.5){\d}&
{\;{\Om^{1}}\;} \ar@<1ex>[r]^<>(0.5){\d}\ar[l]^<>(0.5){\b}&
{\;{\Om^{2}}\;} \ar@<1ex>[r]^<>(0.5){\d}\ar[l]^<>(0.5){\b}&
\ar[l]^<>(0.5){\b}\ldots.
}
$$

Here, the de Rham differential $\d$ and  the  Hochschild differential $\b$,
defined in \eqref{b}, are related via an important {\em  Karoubi
operator}
 $\ka: {\Om^\hdot  }\to{\Om^\hdot  }$ \cite{Ka}.
The latter is defined by the formula
$\ka:\al\,\d a\mto(-1)^{\deg\al}\,\d a\,\al\,$ if $\deg\al>0$,
and $\ka(\al)=\al$ if $\al\in {\Om^0 }$.
By \cite{Ka},\cite{CQ1}, one has
$$\b\ccirc\d +\d\ccirc \b=\Id-\ka.$$

It follows that $\ka$ commutes with both $\d$ and $\b$.  Furthermore,
it is easy to verify (see \cite{CQ1} and the proof of Lemma
\ref{ka_iso} below) that the Karoubi operator descends to a
well-defined map $\ka: \Om^n  \br\to\Om^n  \br$, which is
essentially a cyclic permutation; specifically, in $\Om^n  \br$, we
have
$$\ka(\,\al_1\,\al_2\,\ldots\,\al_{n-1}\,\al_n\,)=
(-1)^{n-1}\al_n\,\al_1\,\al_2\,\ldots\,\al_{n-1},\quad\forall
\al_1,\ldots,\al_n\in\Om^1.
$$

Let $(-)^\ka$ denote taking $\ka$-invariants.
In particular,  write $(\Om^\hdot   )^\ka\br:=
[(\Om^\hdot  )\br]^\ka\sset 
(\Om^\hdot   )\br$. 

\begin{prop}\label{bkappa} For any $n\geq 1$, we have
an equality
$$\bi_\Delta=(1+\ka+\ka^2+\ldots+\ka^{n-1})\ccirc \b
\quad\op{as}\;\;\op{maps}\en{{\Om^n  }}\;\to\;{\Om^{n-1}  }.
$$

Furthermore, the map $\bi_\Delta$ fits into a canonical
short exact sequence
$$0\map H^n(\Om^{\,}\a,\b)\map \DR^n A\stackrel{\bi_\Delta}\too 
[A,\Om^{n-1} A]^\ka\map 0.
$$
\end{prop}

We recall that the cohomology group $H^n(\Om^{\,}\a,\b)$ that occurs in
the above displayed short exact sequence is isomorphic, as has been
mentioned in \S\ref{hoch_sect}, to the Hochschild homology $H_n(A,A)$.
Thus, Theorem \ref{hoch_thm} is an immediate consequence of the short
exact sequence of the proposition.
\medskip

\underline{\sc{Special case:} $H_1(A,A)$.}\en\
For one-forms, the formula of Proposition \ref{bkappa} gives
$\bi_\Delta=\b$. Thus, using the
identification $H_1(A,A)=H^1(\omh, \b)$,
 the short exact sequence of  Proposition \ref{bkappa}
reads
\begin{equation}\label{b1}
0\map H_1(A,A)\map \DR^1 A\,\stackrel{\b=\bi_\Delta}\tooo\, [A,A]\map 0.
\end{equation}

The short exact  sequence \eqref{b1} may be obtained in an
alternate way as follows.
We apply  the right exact functor $(-)\br$
to \eqref{fund}.
The corresponding  long  exact sequence of Tor-groups
reads
$$\ldots\to H_1(A,A\otimes A)
\to H_1(A,A)\to(\Om^1 )\br\to
(A\otimes A)\br\stackrel{c}\to A\br\to 0.
$$
Now, by the definition of  Tor, 
$H_k(A,A\otimes A)=0$ for all ${k>0}$.
Also, we have natural
 identifications $(\Om^1 )\br=\DR^1 A$
and
$(A\otimes A)\br\cong A$.
This way, the  map $c$ on the right of the displayed formula
above may be identified with
the natural 
projection
$A\onto A/[A,A]$. Thus, $\Ker(c)=[A,A]$,
and the long  exact sequence above
reduces to the short exact sequence \eqref{b1}.

It is immediate from definitions that the
map $\b=\bi_\Delta$ in \eqref{b1} is given
 by Quillen's
formula $
u\,\d v\mapsto [u,v]$ \cite{CQ1}.
In particular, we deduce that
$\dis (\DR^1 A)\exact\sset
\Ker(\bi_\Delta)=H_1(A,A).
$

\subsection{Proof of  Proposition \ref{bkappa}}\label{pf1}
We first state and prove
a lemma which was implicit
in  \cite{CQ2} and \cite[\S 2.6]{L} and which will play an important role
in \S \ref{pf_main} below.

\begin{lem}\label{ka_iso} \vi The projection
  $(\Om^\hdot  )\br\to\DR^\hdot A $
restricts to a {\textbf{bijection}}
$(\Om^\hdot   )^\ka\br\iso\DR^\hdot  A $.

\vii The map $\b$ descends to 
a map $\b\br: (\Om^\hdot   )\br\to
{\Om^{\hdot-1} }$.

\viii  The kernel of the map  $\b\br: (\Om^\hdot   )^\ka\br\to
{\Om^{\hdot-1} }$,
the restriction of $\b\br$ to the space of
$\ka$-invariants, is isomorphic to $H^n(\Om^{\,}\a,\b)$.
\end{lem}

\begin{proof}[Proof of Lemma] 
The argument below follows
the proof of \cite[Lemma 2.6.8]{L}.

From definitions, we get
$[A,\Om]=\b\Om$ and $[\d A,\Om]=(\Id-\ka)\Om$.
Hence,  we obtain, cf.~\cite{CQ1}:
$$[\Om,\Om]=[A,\Om]+[\d A,\Om]=\b\Om+(\Id-\ka)\Om.
$$
 We deduce 
that
$\Om\br=\Om/\b\Om$ and
$\dis\DR^\hdot A =\Om/[\Om,\Om]=\Om\br/(\Id-\ka)\Om\br$.
In particular, since $\b^2=0$, the map $\b$ descends to a well
defined map $\b\br: \Om\br=\Om/\b\Om\to\Om$.

Further, one has the following
standard identities \cite[\S2]{CQ2}, on $\Omega^n$ for all $n \geq 1$:
\begin{equation}\label{ident} \ka^n-\Id=\b\ccirc\ka^n\ccirc \d,
\qquad\ka^{n+1}\ccirc \d=\d.
\end{equation}

The  Karoubi operator $\ka$ commutes with $\b$, and
hence induces a well-defined endomorphism of
 the vector space
 $\Om^n/\b\Om^n,\, n=1,2,\ldots$. Furthermore,
from the first identity in \eqref{ident}
we see that $\ka^n=\Id$ on
 $\Om^n/\b\Om^n$.
Hence, we have a direct sum decomposition
$\Om\br=(\Om\br)^\ka\oplus(\Id-\ka)\Om\br$.
It follows that  the natural
projection $\Om\br=\Om/\b\Om\onto \DR^\hdot A =\Om\br/(\Id-\ka)\Om\br$
restricts to an isomorphism
$(\Om\br)^\ka\iso\DR^\hdot A$.
Parts (ii)--(iii) of  Lemma  \ref{ka_iso} are
clear from the proof of \cite{L}, ~Lemma 2.6.8. \end{proof}

\begin{proof}[Proof of Proposition \ref{bkappa}.]
 The  first statement  of the proposition is immediate from
the formula of Lemma \ref{biDelta}(i).
To prove the second statement, we 
exploit the first identity in \eqref{ident}. Using
the formula for $\bi_\Delta$ and the fact
that $\b$ commutes with $\ka$, we compute that
\begin{equation} \label{km1iD}
(\ka-1)\ccirc\bi_\Delta=\b\ccirc(\ka-1)\ccirc(1+\ka+\ka^2+\ldots+\ka^{n-1})
=\b\ccirc(\ka^n-1)=\b^2\ccirc\ka^n\ccirc \d=0.
\end{equation}

Hence, we deduce that the image of $\bi_\Delta$ is contained
in $(\b\Om)^\ka$.
Conversely, given any element $\al=\b(\beta)\in (\b\Om)^\ka$,
we find that
$$\bi_\Delta(\beta)=
(1+\ka+\ka^2+\ldots+\ka^{n-1})\ccirc\b\beta=n\cdot\b\beta=n\cdot\al.
$$
Thus, 
$\op{Im}(\bi_\Delta)=(\b\Om)^\ka=[A,\Om]^\ka$, since $\b\Om={[A,\Om]}$.
Furthermore, since
$(1+\ka+\ka^2+\ldots+\ka^{n-1})\ccirc \b = n
\b$ on
$(\Om^\hdot)^\ka\br$, the two maps
have the same kernel.
The exact sequence of the proposition  now
follows from Lemma \ref{ka_iso}.
\end{proof}

\subsection{Harmonic decomposition}  \label{harms}
Our proof of Theorem \ref{main} is an adaptation of the strategy used in
\cite[\S2]{CQ2}, based on the {\em harmonic decomposition}
\begin{equation} \label{harmdec}
\bom=P\bom\oplus P^\perp\bom,\quad\text{where}\quad
P\bom:=\Ker(\Id-\ka)^2,\quad P^\perp\bom:=\op{Im}(\Id-\ka)^2.
\end{equation}

The differentials $\B, \b$, and $\d$ commute with $\ka$, hence
preserve the harmonic decomposition.

It will be convenient to introduce  two 
degree preserving linear maps  ${\mathsf N},\ {\mathsf N}!:\
\bom\to\bom$, such that, for any $n\geq 0$, 
\begin{equation}\label{Ndfn}
{\mathsf N}|_{\bom^{_n}} \en \text{ is multiplication by $n$},\quad\text{and}
\quad {\mathsf N}! |_{\bom^{_n}} \en\text{ is multiplication by $n!$.}
\end{equation} 
Then, 
\begin{equation}\label{agree}
{\mathsf{(i)}}\en
\B={\mathsf N}\,\d\, P,
\qquad\text{and}\qquad
{\mathsf{(ii)}}\en\bi_\Delta=\b\, {\mathsf N}\, P.
\end{equation}

Here, equation (i) is proved  by Cuntz and Quillen
\cite[\S 2, formula (11)]{CQ2}
using  the second identity in
\eqref{ident}.

\begin{proof}[Proof of \eqref{agree}{\em (ii)}] First,
 we use that $\b$ commutes with $\ka$.
Therefore, applying \eqref{km1iD},  we find
$i_\Delta \ccirc (\Id-\ka)^2 =$ $(\ka-1) \ccirc i_\Delta \ccirc (\ka - 1) =
0$. We conclude that the operation $\bi_\Delta$
annihilates $P^\perp\bom$.

It remains to show that, on $P\bom^n$, one has
 $\bi_\Delta=(n-1)\cdot\b$. To this end,
 let $\al\in \bom^{_n}$. From the first  identity in
\eqref{ident}, $\al-\ka^n(\al)\in\b\bom$.
Hence, 
$\b\al-\ka^n(\b\al)\in \b^2\bom=0$,
 since $\b^2=0$.
Thus, the operator $\ka$ has finite order on
$\b\bom$, and hence on $\b(P\bom)$.
 But, for any operator $T$ of finite order,
$\Ker(\Id-T)=\Ker((\Id-T)^2)$.
It follows  that, if $\al\in P\bom^{_n}$,
then $\b\al\in\Ker((\Id-\ka)^2)=\Ker(\Id-\ka)$.
We conclude that the element $\b\al$ is
fixed by $\ka$. Hence,
$(1+\ka+\ka^2+\ldots+\ka^{n-1})\ccirc \b\al=
n\cdot \b\al$. Therefore, by Proposition \ref{bkappa},
$\bi_\Delta(\al)=
n\cdot \b\al$,
and \eqref{agree}(ii) is proved. 
\end{proof}

\subsection{Proof of Theorem \ref{main}}\label{pf2}
Since the harmonic decomposition is stable under
all four differentials $\B, \b, \d$, and $\bi_\Delta$,
we may analyze the homology of each of the direct summands,
$P\bom$ and $P^\perp\bom$, separately. 

First of all, we know that  $\B=0$ on $P^\perp \Omega$,
by \eqref{agree}(i), and moreover
it has been shown by Cuntz and Quillen \cite[Proposition 4.1(1)]{CQ2}
that $(P^\perp \Omega, \b)$ is acyclic.
Furthermore, since the complex $(\bom,\d)$ is acyclic (see \cite[\S1]{CQ2} or
\cite{CBEG} formula (2.5.1)), we deduce the
following. 
\begin{equation}\label{acy} \text{\em Each of
  the complexes}\en
(P\bom,\d)\en \text{\em and} \en (P^\perp\bom,\d)\en \text{\em is
  acyclic.}  \end{equation}

Now,
the map $\bi_\Delta$ vanishes on
$P^\perp\bom$
by \eqref{agree}(ii). Hence, on $P^\perp\bomt$, we have
$\d+t\cdot \bi_\Delta=\d$.
 Therefore, we conclude using \eqref{acy}
that  $(P^\perp\bom[t], \d)$, and hence also $(P^\perp\bom[t], \d+t\cdot
 \bi_\Delta)$, are
acyclic complexes.

Thus, to complete the proof of the theorem, we must compare cohomology
of the complexes $\dis(P\bomt, \d+t\cdot \bi_\Delta\bigr)$ and
$\dis(P\bomt, \B+t\cdot \b)$.  We have $\dis {\mathsf N} \cdot\d +
({\mathsf N}+1)^{-1} \cdot t \cdot \bi_\Delta = \B + t\b.  $
Post-composing this by ${\mathsf N}!$ (see \eqref{Ndfn}), we obtain
$\dis ({\mathsf N}!) \cdot (\d + t \cdot \bi_\Delta) = (\B + t \cdot
\b) \cdot ({\mathsf N}!).  $ We deduce the following isomorphism of
complexes which completes the proof of the theorem:
$${\mathsf N}!:\
(P\bomt,\, \d+t\cdot \bi_\Delta\bigr)\ \iso \
(P\bomt,\, \B+t\cdot \b). \eqno\Box$$

\subsection{Negative and ordinary cyclic homology}\label{ss:nch}
It is possible to extend  Theorem \ref{main}
 to the case of (nonperiodic) cyclic homology and negative cyclic homology
 using harmonic decomposition.
To explain this,  put
$$
(\bomt)_{_+}\ :=\  \underset{m \geq 0}\bplus \ \left(\prod\nolimits_{i < m}\
 t^i \Omega^{m-2i}\right), \qquad
(\bomt)_{_{\geq 0}} \ :=\  \underset{m \geq 0}\bplus \ \left(\prod\nolimits_{i \leq
m}\
 t^i \Omega^{m-2i}\right).
$$

It follows from definitions that  cyclic homology and
negative cyclic homology, respectively, may be defined in terms of the following complexes (see, e.g., \cite[Chapters 2--3]{L}):
\begin{equation}\label{hchc}
HC_\idot = H^{-\hdot}\big(\bomt / (\bomt)_{_{+}},\ \B + t \cdot \b\big), \quad
HC^-_\idot = H^{-\hdot}\big((\bomt)_{_{\geq 0}},\ \B + t \cdot \b\big).
\end{equation}

Furthermore, we introduce two other homology theories, ${}^\heartsuit\!
HC,\ {}^\heartsuit\! HC^-$. The corresponding
homology groups are defined as the homology groups of
complexes similar to \eqref{hchc}, but where the differential
 $\B + t \cdot \b$ is replaced by $\d+t  \cdot \bi_\Delta$.

 Now, in terms of the projection to the harmonic part (recall
 \eqref{harmdec}),  we have the following.
\begin{prop} There are natural graded  $\k[t]$-module isomorphisms
$$
P({}^\heartsuit\! HC_\idot) \cong HC_\idot \quad\oper{and}\quad
P({}^\heartsuit\! HC^-_\idot) \cong HC^-_\idot.
$$
\end{prop}
\begin{proof} There are natural  splittings 
{\small
$$\bomt = \frac{\bomt}{(\bomt)_{_+}}\ \bplus\ (\bomt)_{_+},\quad
  \bomt = \frac{\bomt}{(\bomt)_{_{\geq 0}}}\ \bplus\ (\bomt)_{_{\geq 0}}.$$}
These splittings
are stable under the differential $t \cdot \b$.
It follows  that
  $P^\perp HC = 0 = P^\perp HC_-$, since $\B = 0$ on $P^\perp \bom$,
  and  the differential $t \cdot \b$ is acyclic here.

 Observe further that, while $\bi_\Delta$ is zero on $P^\perp \bom$, the above
  splittings do \emph{not} stabilize $\d$.  So, we can pick up some
  nonzero groups
  $P^\perp ({}^\heartsuit\! HC)$, or
$P^\perp ({}^\heartsuit\! HC^-)$.
  However, restricting to the harmonic part, the proof of Theorem \ref{main}
  yields an isomorphism of complexes
 $$\big(P[\bomt / (\bomt)_{_+}],\ \B + t \cdot \b\big)\
\cong\
\big(P[\bomt / (\bomt)_{_+}],\ \d + t \cdot \bi_\Delta\big),$$
 and
also a  similar isomorphism involving $(\bomt)_{_{\geq 0}}$.  
\end{proof}
\begin{rem}
  In view of these results, it is reasonable to ask why our
  construction of the Gauss-Manin connection does not extend to
  ordinary cyclic homology and negative cyclic homology. The point is
  that, in these cases, one must replace $\Omega^{i}\com B$ in
  \eqref{grass} by $t^{-\lfloor i/2 \rfloor} \Omega^{2i}\com B$, and
  so one does not obtain a connection over $B$.  (Note also that the
  results of this subsection require passing to the harmonic part,
  which does not seem to fit well into our construction of the
  Gauss-Manin connection.)
\end{rem}

\subsection{Proof of Theorem \ref{gmthm}}\label{pfgm} Let $B\sset A$ and assume 
that   $A/B$ is a
free $B$-module.  The main step of the proof is the following.

\begin{lem}\label{frlem}  Let
$\{a_s,\ s\in {\scr S}\}$ be a basis  of $A/B$ as a free
 $B$-module. Then, the  elements below form a basis for $\Omega^B(A)$ as a free
$\Omega^\hdot(B)$-module:
\begin{equation}\label{ombasis}
a_{s_0} \d a_{s_1} \d a_{s_2} \cdots \d a_{s_m}, \quad \d a_{s_1} \d a_{s_2} \cdots
\d a_{s_m},\quad
s_j\in\scr S.
\end{equation}
\end{lem}

\begin{proof}
  Observe that the short exact sequence $B\to A\to A/B$ splits as a
  sequence of $B$-modules, since $A/B$ is a free $B$-module. By abuse
  of notation we will view $A/B$ as a subspace of $A$ via a fixed
  choice of such splitting.  Hence, the elements $1$ and $\{a_s,\ s\in
  {\scr S}\}$ give a $B$-module basis of $A$.

Now, $\Omega^m A \cong A \otimes (A/\k)^{\otimes m}$ is spanned over $\k$ 
by elements of the form
\begin{equation}\label{e:oma-basis}
f_0 \d f_1 \cdots \d f_m,
\end{equation}
where each $f_i$ is of the form $b a_s$ for some $b \in B$ and $s \in \mathscr S$, except for $f_0$ which can also be an element of $B$ itself.

Now, let $\Omega_{\mathscr S} \subseteq \Omega A$ be the subspace spanned by
elements of the form \eqref{ombasis}.  Let $\Omega' := \Omega B \otimes_\k
\Omega_{\mathscr S}$ be the free $(\Omega B)$-module it generates.  We
have a projection $\pi: \Omega A \onto \Omega'$, given on elements
\eqref{e:oma-basis} by rewriting $\d (b a_s) = (\d b) a_s + b (\d
a_s)$, and then moving all $b$ and $\d b$ terms to the front of the
expression (via supercommutators). It is easy to see that this is
well-defined.
Thus, Lemma \ref{frlem} reduces to the next claim, which
we prove below.
\end{proof}
\begin{claim}\label{cl:frlem}
The map $\pi$ descends to an isomorphism $\pi': \Omega^B A \iso \Omega'$.
\end{claim}
\begin{proof}
  Consider the multiplication map $\mu: \Omega' =\Omega B \otimes_\k
  \Omega_{\mathscr S} \to \Omega A$. It is easy to see that $\mu$ is
  injective and that the composition $\pi \circ \mu$ is the identity.
  This realizes $\Omega'$ as a direct summand of $\Omega A$.  Let
  $\pr: \Omega A \onto \Omega^B A$ be the projection.  Note that the
  composition
\[
\displaystyle \Omega A \mathop{\to}^\pi \Omega' \mathop{\to}^{\mu} \Omega A \mathop{\to}^{\pr} \Omega^B A
\]
coincides with $\pr$, by the definition of $\pi$.  

We will show that $\pi$ factors through $\pr$, i.e., $\pi = \pi' \circ \pr$
for $\pi': \Omega^B A \to \Omega'$.
 From this and the above, it follows that $(\pr \circ \mu)
\circ \pi' = \Id$.  Since $\pi'$ is
surjective, $(\pr \circ \mu)$ and $\pi'$ are therefore inverse isomorphisms.

Equivalently, we need to show that $\pi$ annihilates the kernel of
$\Omega A \onto \Omega^B A$, i.e., the two-sided ideal generated by
$[B, \Omega A]$ and $[\d B, \Omega A]$.  We first prove the following
multiplicative property of $\pi$:
\begin{equation}\label{e:pi-mult}
  \pi(\alpha \beta) = \pi \bigl( (\mu \circ \pi)(\alpha) (\mu \circ \pi)(\beta) \bigr).
\end{equation}
To prove this, it is enough to assume that $\alpha$ and $\beta$ are
elements of the form \eqref{e:oma-basis}.  Say $\beta = f \d \beta'$
where $\beta'$ is also of this form.  Then, we need to show that
$\pi(\alpha \cdot f \cdot \d \beta') = \pi(\pi(\alpha) \cdot f \cdot
\pi(\d \beta'))$.  Since, in either expansion, each term $\d (b a)$ in
$\d \beta'$ is replaced by $(\d b) a + b (\d a)$, and then after
everything else is so expanded, the $\d b$ and $b$ are moved to the
front, it follows that one can ignore the $\d \beta'$ and merely prove
that $\pi(\alpha \cdot f) = \pi(\alpha) \cdot f$.  Now, to expand
$\pi(\alpha \cdot f)$, we need to use
\begin{multline}\label{e:bar-exp}
\d f_1 \cdots \d f_m \cdot f_{m+1} \\ = (-1)^m f_1 \d f_2 \cdots \d f_{m+2}
 + \sum_{i=1}^m (-1)^{m-i} \d f_1 \cdots \d f_{i-1}
\d (f_i f_{i+1}) \d f_{i+2} \cdots \d f_{m+1}.
\end{multline}
So, we have to show that applying \eqref{e:bar-exp} and then applying
$\pi$ is the same as applying $\pi$ to $\alpha$, multiplying by $f$,
and applying \eqref{e:bar-exp}.  This is a straightforward
verification.

Now, using \eqref{e:pi-mult}, the claim reduces to showing that
$\pi([B,\d A]) = \pi([\d B, A]) = \pi([\d B, \d A]) = 0$.  The first
two identities are equivalent, since $\d [B, A] = 0$, and are
straightforward to verify.  It is immediate that $\pi([\d B, \d A]) =
0$. So this reduces the claim to \eqref{e:pi-mult}. \qedhere
\end{proof}

It is instructive, for the proof of Theorem \ref{gmthm} presented below,
 to have  in mind the situation of Remark \ref{getzrem}, where
  $B = \k\lll x_1, \ldots, x_n \rrr$ and
  $A = (A_0 \lll x_1, \ldots, x_n \rrr,\ \star)$. Then, $\{a_s\}$ is a
  $\k$-basis of $A_0$, and the (topologically-free version of)
  Lemma \ref{frlem}  becomes more obvious.  

\begin{rem}  More generally, given an arbitrary {\em regular} commutative algebra
  $B$ and a maximal ideal $\mathfrak{m} \subset B$, taking the
  $\mathfrak{m}$-adic completions $\wh{A}_{\mathfrak{m}}$ and
  $\wh{B}_{\mathfrak{m}}$ reduces to the above situation.
\erem

\begin{proof}[Completion of the proof of Theorem \ref{gmthm}]
We take $B$ as a ground ring and apply Theorem \ref{main}
(which applies over any commutative base ring 
containing $\k$).
We deduce that
 $$H^{-i}\big(\bom(A;B)\ttc,\
\d + t \bi_\Delta\big)
\ \cong\
HP_\idot^B(A), \qquad\forall i\in \Z.
$$

Now Lemma \ref{frlem} implies that the map in \eqref{grass} is an
isomorphism, since the basis for $\Omega^B A$ as a free
$\Omega^\hdot(B)$-module is also a basis for the associated graded
$\gr^i_F \oba$, and a $B$-module basis for
$\Omega^\hdot(A;B) = \gr^0_F \oba$.
The construction of the Gauss-Manin connection given in \S \ref{gmsec}
completes the proof of
the theorem.
\end{proof}

\begin{proof}[Proof of Corollary \ref{getzex}] We only need to show that,  in the present setting, the map in
\eqref{grass} is an isomorphism. For
that, we observe that the argument used in the proof
 goes through provided
that  $A$ is only {\em topologically} free over $B$, and our claim follows.
\end{proof}

\section{The Representation functor}\label{repf}
\subsection{Evaluation map}\label{ev_sec}
We fix a finite-dimensional $\k$-vector space $V$. Set
$\oEnd := \End_\k(V)$.  Given affine schemes $X$ and $S$, let
$X(S) = \Hom(S, X) = \Hom_{\k\text{-}\mathsf{alg}}(\k[X], \k[S])$ denote the
set of $S$-points of $X$.  

Given an algebra $A$, we may consider the set
$\Hom_{\k\text{-}\mathsf{alg}}(A,\oEnd)$ of all algebra maps
$\rho: A\to \oEnd$.  More precisely, to any {\em finitely
  presented}
associative $\k$-algebra $A$ we associate an affine scheme of finite
type over $\k$, to be denoted $\Rep(A,V)$,
such that
$\Rep(A,V)(B) \cong \Hom_{\k\text{-}\mathsf{alg}}(A,B\o \oEnd)$.  That
is, the $B$-points of $\Rep(A,V)$ correspond to families of
representations of $A$ parametrized by $\Spec B$.  Write
$\k[\Rep (A,V)]$ for the coordinate ring of the affine scheme
$\Rep (A,V)$, which will be always assumed to be {\em nonempty}.

The tensor product $\oEnd\o\k[\Rep
(A,V)]$ is an associative algebra of polynomial 
maps $\Rep(A,V) \rightarrow \oEnd$.
To each element $a\in A$, we associate the element $\wh{a} \in \oEnd
\o \k[\Rep(A,V)]$, which on the level of points, has the form $\dis\wh{a}:
\Rep(A,V)(B) \rightarrow \oEnd \o B,\ \wh{a}(\rho) = \rho(a)$.  This yields an algebra
homomorphism, called the {\em evaluation map}
$\ \dis\ev:\ A\too\oEnd \o \k[\Rep (A,V)],\ a\mto\wh{a}.
$

\subsection{Extended de Rham complex and equivariant cohomology}
Let $X$ be an affine scheme with coordinate ring $\k[X]$, tangent
module $\calT(X):=\Der(\k[X],\ \k[X])$, and module of K\"ahler
differentials $\Om^1(X)$. We write  $\Omega^\hdot(X)=
\La^\hdot_{\k[X]}\Om^1(X)$ for the
DG algebra of differential forms, equipped with the de Rham differential
$d_X$.

Let $\g g$ be a finite-dimensional Lie algebra, and let $\g g$ act on
$\k[\g g]$, the polynomial algebra on the vector space $\g g$, by the
adjoint action. We view $\k[\g g]$ as an even-graded algebra such that
the vector space of linear functions on $\g g$ is assigned degree 2.

Given a Lie algebra map $\g g\to\calT(X),\ e\mto\arr{e}$, we get a
$\g g$-action $\om\mto L_{\arr{e}}\om$ on $\Omega^\hdot(X)$, by the
Lie derivative.  This makes the tensor product
$\Omega^\hdot(X,\,\g g):=\Omega^\hdot(X)\otimes\k[\g g]$ a graded
algebra, equipped with the total grading and with the $\g g$-diagonal
action. Let $i_{\arr{e}}$ denote the contraction. Then, set
\begin{equation}
d_\g g:\ \Omega^\hdot(X,\,\g g)\too\Omega^{\hdot+1}(X,\,\g g),\quad
\om\o f\lto \sum\nolimits_{r=1}^{\dim \g g}\ (i_{{\arr{e}_r}}\om)\o
(e_r^*\cd f), \label{dgdfn}
\end{equation}
where $\{e_r\}$ and $\{e_r^*\}$ stand for dual bases of $\g g$ and
$\g g^*$, respectively.  This map restricts to a differential $d_\g g$
on $\Omega^\hdot(X,\,\g g)^{\g g}$, the graded subalgebra of $\g
g$-diagonal
invariants.

\begin{defn}
  A differential form $\om\in\Omega^\hdot(X)$ is called {\em basic}
  if, for any $e\in \g g$, both $L_{\arr{e}}\om=0$ and
  $i_{\arr{e}}\om=0$. Basic forms form a subcomplex
  $\Omega^\hdot\bas(X)\sset\Omega^\hdot(X)$ of the de Rham complex.

  Furthermore, define the $\g g$-{\em equivariant algebraic de Rham
    complex}
  of $X$ to be the complex \begin{equation}\label{dr_equiv} \bigl(\Omega^\hdot(X,\,\g
  g)^{\g g},\,d_{\text{DR}}+d_\g g), \qquad d_{\text{DR}}:=d_X\o
  \id_{\k[\g g]}.
\end{equation}
\end{defn}

We now return to the setup of \S\ref{ev_sec}. Thus we fix a
finitely-presented algebra $A$, a finite-dimensional vector space $V$,
and consider the scheme $\Rep(A,V)$.

Let  $G = \GL(V)$. This is an algebraic group over $\k$ that
 acts naturally on the algebra
$\oEnd$ by inner automorphisms, via conjugation. Hence, given an
algebra homomorphism $\rho: A\to \oEnd$ and $g\in G(\k)$, one
has a conjugate homomorphism
$g(\rho): a\mapsto g\cdot\rho(a)\cdot g\inv$. Then, the action
$\rho\mapsto g(\rho)$ makes $\rep$ a $G$-scheme (extending in the
obvious way to $B$-valued representations for any $B$).

Let $\g g:=\Lie G$ be the Lie algebra of $G$.
The action of  $ G$ on $\rep$ induces
a Lie algebra  map 
\begin{equation}\label{actx}
\act_A:\ \g g \map\calT(\Rep(A,V)),
\quad x\lto \arr{x}=\act_A(x).
\end{equation}
Thus, we may consider $\Omega^\hdot(\Rep(A,V),\,\g g)^{\g g}$,
 the corresponding 
 $\g g$-equivariant algebraic de Rham complex.

Now, thanks to Lemma \ref{biDelta}(ii), the map 
$\d+i_\Delta: \DR_t A\to\DR_t A$ squares to zero.
We call the resulting complex $\big(\DR_t A,\ \d+i_\Delta\big)$
the {\sf noncommutative equivariant de Rham complex}.
The second isomorphism of the following theorem,
which is the main result of this section,
shows that this complex is indeed a noncommutative
analogue of the equivariant de Rham complex \eqref{dr_equiv}.

 Recall the operator ${\mathsf N}$ from \eqref{Ndfn}.

\begin{thm}\label{rep_thm}
The evaluation map induces the following  canonical morphisms of complexes:
\begin{align}
\bigl(H_\idot(A,A),\ \B\bigr)&\stackrel{\ev}\too
\bigl(\Omega^\hdot\bas(\Rep(A,V)),\ ({\mathsf N} + 1)\ccirc d_{\text{DR}}\bigr),\quad\oper{and}\label{first}\\
\bigl(\DR^\hdot_t A,\, \d+i_\Delta\bigr)&\stackrel{\ev}\too
\bigl(\Omega^\hdot(\Rep(A,V),\,\g g)^{\g g},\ d_{\text{DR}}+d_\g
g\bigr).\label{second}
\end{align}
\end{thm}

We begin the proof with some general constructions. 

\subsection{Evaluation map on differential forms}
Observe that giving an algebra homomorphism $\rho: \k[t]\to B\o \oEnd$
amounts to specifying an arbitrary element $x=\rho(t)\in B \o \End$.
Thus,  $\Rep(\k[t],V)(B) = B \o \End$.

Similarly, for any algebra $A$, giving an algebra morphism
$\rho:A*\k[t]$ $\to B \o \oEnd$ amounts to giving a homomorphism
$A\to B \o \oEnd$ and an arbitrary additional element
$x=\rho(t)\in B \o \oEnd$.  We see that
$\Rep(A_t,V) \cong \Rep(A,V) \times \End$. Let $\pi$ denote
the second projection, which is $G$-equivariant.  We will use
shorthand notation
$$\Rep_t:=\Rep(A_t,V)=\Rep(A,V)\times\End,\quad\text{and}\quad\Rep:=\Rep(A,V).
$$

Let
$\Om^\hdot_\pi(\Rep_t)$
be the DG algebra of relative (with respect to $\pi$)
 algebraic differential forms on
the scheme $\Rep_t$ (in the ordinary sense of commutative algebraic
geometry).
By definition, 
\begin{equation}\label{rel}
\Om^\hdot_\pi(\Rep_t):=\La^\hdot_{\k[\Rep_t]}\,\Om^1_\pi(\Rep_t) \cong
\Om^\hdot(\Rep)
\otimes \k[\End].
\end{equation} 

Generalizing the construction of \S\ref{ev_sec},
 we now introduce an evaluation map
 on relative
differential forms. 
In more detail, given $n=0,1,2,\ldots$,
write $m:$
 $\oEnd^{\otimes (n+1)}\to\oEnd$ for the $n$-fold
multiplication map.
We define a map $\evom$  as the following composite:
\begin{align*}
\Om^n_{\k[t]}(A_t) =A_t\otimes (A_t/\k[t])^{\otimes n}&\ \stackrel{\ev}\to\
(\oEnd \o\k[\Rep_t])\otimes
\bigl(\oEnd \o\Om^1_\pi(\Rep_t)\bigr)^{\otimes n}\\
&\ \to\
\oEnd^{\otimes n+1} \o \bigotimes
\bigl(\La_{\k[\Rep_t]}^n\,\Om^1_\pi(\Rep_t)\bigr)
\stackrel{m\o\Id}\tooo
\oEnd \o\Om^n_\pi(\Rep_t).
\end{align*}

Any element in the image of this composite  is easily seen to
be $G$-invariant with respect to the $G$-diagonal action on
$\oEnd \o\Om^n_\pi(\Rep_t)$. Thus, the composite above yields
 a well-defined, canonical DG  algebra map
$$\evom:
\Om_t A\to
\bigl(\oEnd\o\Om^\hdot_\pi(\Rep_t)\bigr)^G\!,
\en\;
\al=a_0\,\d a_1\ldots \d a_n\mapsto \wh{\al}=\wh{a}_0\ \ddr \wh{a}_1\ldots  \ddr \wh{a}_n.$$

Furthermore, we have the linear function
 $\Tr: \oEnd \to\k, \,x\mapsto\Tr(x)$.
We  form the composite
\begin{equation}\label{phi}\Om_t A\stackrel{\evom}\too
\bigl(\oEnd\o\Om^\hdot_\pi(\Rep_t)\bigr)^G\stackrel{\Tr\o\Id}\tooo
 \bigl(\k\o\Om^\hdot_\pi(\Rep_t)\bigr)^G=\Om^\hdot_\pi(\Rep_t)^G,
\quad \al\mapsto
\Tr\wh{\al}.
\end{equation}

The above composite  vanishes on the (graded) commutator space
$[\Om_t A,\Om_t]$ $\sset\Om_t A$, due to symmetry of the trace function.
Therefore, the map in \eqref{phi}
descends to $\DR^\hdot(\Om_tA)$.

We remark next that the Lie algebra $\g g=\Lie G$ is nothing but the
associative algebra $\oEnd$ viewed as a Lie algebra.  Hence, using the
isomorphisms in \eqref{rel}, we can write
$$\Om^\hdot_\pi(\Rep_t)=\Om^\hdot(\Rep)\o\k[\End]=\Om^\hdot(\Rep)
\o\k[\g g]=\Om^\hdot(\Rep,\,\g g).$$
Thus, by the definition of the extended de Rham complex, $\DR^\hdot_tA$,
the composite in \eqref{phi}  gives a map
\begin{equation}\label{trace_om2}
(\Id\otimes\Tr)\ccirc\evom:\
\DR^\hdot_tA\map \Om^\hdot_\pi(\Rep_t)^G=
\Om^\hdot(\Rep,\,\g g)^{\g g}.
\end{equation}

\subsection{Proof of Theorem \ref{rep_thm}} 
It is clear that $\d$ is
 taken to $\ddr $ under \eqref{trace_om2}. Hence,
proving \eqref{second}, where the map `$\ev$' stands for
 $(\Id \o \Tr) \ccirc \evom$, amounts to  showing commutativity
of the diagram
\begin{equation}\label{diag}
\xymatrix{
\DR_tA\ar[d]_<>(0.5){i_\Delta}\ar[rrr]^<>(0.5){(\Id\otimes\Tr)\ccirc\evom}&&&
\Om(\Rep,\,\g g)^{\g g}\ar[d]_<>(0.5){d_{\g g}}\\
\DR_tA\ar[rrr]^<>(0.5){(\Id\otimes\Tr)\ccirc\evom}&&&
\Om(\Rep,\,\g g)^{\g g}.
}
\end{equation}

To see this, we note that, for any $a_0, \ldots, a_n \in A_t$,
one has
\begin{equation} \label{evomeqn}
  \ev_\Omega \ccirc i_{\Delta} [a_0\, \d a_1\, \d a_2 \ldots \d a_n] 
  = (\Id \o \Tr) \Bigl(\sum_{i=1}^n \wh{a}_0\  \ddr \wh{a}_1
\ldots  \ddr \wh{a}_{i-1}\ [\wh{t}, \wh{a}_i]\  \ddr \wh{a}_{i+1} 
\ldots  \ddr \wh{a}_n\Bigr).
\end{equation}
Next, note that $\hat t$ may be identified with the element
$\Id \in \g g \o \g g^* \cong \mathrm{End} \o \g g^* \subset \oEnd \o
\k[\g g]$. Furthermore, for any element $e \in \g g$ and any
$s \in \k[\Rep(A, V)] \o \End$, we evidently
have $i_{\arr{e}} (d s) = \ad e(s) = [(1 \o e), s]$. As a consequence, by
\eqref{dgdfn}, we obtain that the RHS of \eqref{evomeqn} may be identified
with $d_\g g((\Id \o \Tr) ({\wh a_0}\d {\wh a_1} \ldots\d {\wh a_n}))$,
as desired.

To prove \eqref{first}, $\ev$ becomes the restriction of
$(\Id \o \Tr) \ccirc \evom$ to
$\ker(i_\Delta) \subset \Omega^\hdot(\Rep)^{\g g}$ (recall Theorem
\ref{hoch_thm}).  Commutativity of \eqref{diag} together with Theorem
\ref{hoch_thm} immediately gives that this induces a morphism
$H_\idot(A, A) \to \Omega^\hdot_{basic}(\Rep(A,V))$. It remains
only to show that $\B$ is carried to $({\mathsf N}+1)\d$.  To see this, we use the
harmonic decomposition \eqref{harmdec}.  Under the quotient
$\bom^{_\bullet} A \onto \DR^\hdot A / \k$ (considering $\k$ to
be the span of the image of $1 \in A$), $P^{\perp} \bom$ is killed,
so the differential $B$ is carried to $({\mathsf N}+1) d_{_\DR}$.  Thus, on $\DR^\hdot A$,
the differential $B$ must reduce to the same as $({\mathsf N}+1) \d$ up to a
scalar.  However, since $\B$ has degree $+1$, the scalar must be zero.
So $\B = ({\mathsf N}+1)\d$ on $\DR^\hdot A$. Thus, the same is true
after passing to $\ker(i_{\Delta})$.

\section{Free products and deformations}\label{ncds}

\subsection{First order deformations based on free
products}\label{defor1}
Recall that, given an associative algebra $A$, we let $A_t=A*\k[t]$ and
write $I=A^+_t=\ldp t\rdp\sset A_t$ for the augmentation
ideal.

A {\em first order free product deformation} of an associative algebra $A$
is the structure of an  associative algebra  on the vector space
$A_t/I^2$ that makes  the vector space $I/I^2\sset A_t/I^2$
a two-sided ideal and that makes the natural
bijection below an algebra isomorphism,
$$\big(A_t/I^2\big)\big/(I/I^2)=A_t/I\ \iso\ A.$$

It is convenient to identify the vector space
$A_t/I^2$ with
$A\oplus (A\o A)$, using \eqref{AtA}.
Thus, we are interested in associative products
on the vector space $A\oplus (A\o A)$ that have
the following form:
\begin{equation}\label{firstass}
\bigl(u\oplus(u'\o u'')\bigr)\; \times\; 
\bigl(v\oplus(v'\o v'')\bigr)\
\stackrel{\starb}\lto\
uv\oplus\bigl(u'\o u''v +uv'\o v'' +\be(u,v)\bigr),
\end{equation}
where $\be: A\times A \to A\o A$ is a certain
$\k$-bilinear map.
 
These products are taken up to an equivalence.
Specifically, for any $\k$-linear map  $f: A\to A\o A$,
we define a linear bijection
$$\tilde{f}:\
A\oplus (A\o A)\too A\oplus (A\o A),\quad
 u\oplus(u'\o u'')\lto u\oplus(u'\o u''+ f(u)).$$
Given a product $\starb$ and
a map $f$, we define a new
 product by transporting the structure via $\tilde{f}$,
that is, by the formula
$\dis
x\star_{_\gamma}y:=\tilde{f}\inv(\tilde{f}(x)\starb\tilde{f}(y))$.
We say that the products $\star_{_\gamma}$ and $\starb$ are
 {\em equivalent}.

A routine calculation, completely analogous
to the  classical one  due to Gerstenhaber, now yields the following.\smallskip

\npb{A first product $\starb$ as in \eqref{firstass}
is associative $\Longleftrightarrow$
$\be\in C^2(A, A\o A)$ is a Hochschild two-cocycle with coefficients
in $A\o A$.}\smallskip

\npb{The products $\starb$ and $\star_{_\gamma}$ corresponding to  two-cochains $\be$ and
$\gamma$ are equivalent $\Longleftrightarrow$
$\be-\gamma$ is a  Hochschild coboundary.}
\medskip

\noindent
Thus,
similarly to   the classical  theory,  we obtain
a   classification of first order star product deformations
in terms of Hochschild cohomology as follows 

\begin{prop} Equivalence classes of 
 associative  products, as  in \eqref{firstass},  are in one-to-one correspondence
with the elements of $H^2(A,A\o A)$, the second  Hochschild cohomology
group
of the $A$-bimodule $A\o A$.\qed
\end{prop}

To study {\em higher order} free product deformations, we have
to introduce first  some new operations on Hochschild cohomology, to be
defined below.

\subsection{Operations on Hochschild cohomology of $A$ valued in the
  tensor algebra of $A$} For any algebra $A$, the natural embedding
$A\into A_t$ makes $A_t$ a graded $A$-bimodule.  Using the
identification \eqref{triv}, we may write $ A_t=\bigoplus_{k\geq 1}
A^{\o k}$. Here, the summand $A^{\o k}$ is assigned degree $2k-2$ and
is equipped with the {\em outer} $A$-bimodule structure defined by the
formula $b(a'\o u\o a'')c:=(ba')\o u\o (a''c)$, for any $a',a'',b,c\in
A$ and $u\in A^{\o (k-2)}$.

Let $C^\hdot(A, A_t)=\bigoplus_{p, k\geq 1}C^p(A,A^{\o k})$ be
the  Hochschild cochain complex with coefficients in the $A$-bimodule $A_t$.
Multiplication in the algebra
$A_t$ induces, for any $p,q,k,m\geq 1$, a cup product
$$
\cup:\ 
C^p(A,A^{\o k})\times C^q(A, A^{\o m})\too
C^{p+q}(A,A^{\o (k+m-1)}),
$$
This way,  $C^\hdot(A, A_t)$ acquires the structure of a bigraded
associative algebra such that  the direct summand $C^p(A,A^{\o k})$ is assigned bidegree
$(p,2k-2)$.
\smallskip

Next, on   $C^\hdot(A, A_t)$, we introduce a pair of new binary operations,
$\vdash$ and $\dashv\,$: 
\begin{gather*}
C^p(A, A^{\o k})\times C^q(A, A^{\o m})\too C^{p+q-1}(A, A^{\o (k+m-1)}),\\
(f, g)\mto
f \vdash g := f^{[1,p]} \ccirc g^{[p,p+q-1]}, \quad\text{and}
\quad (f, g) \mto f \dashv g := g^{[k,k+q-1]} \ccirc f^{[1,p]},
\end{gather*}
where $f^{[i,j]}$ denotes applying $f$ to the consecutive components
$i, i+1, i+2, \ldots, j$, that is, 
$$f^{[i,j]}(a_1 \o \ldots \o a_\ell) =
a_1 \o \ldots \o a_{i-1} \o f(a_i \o \ldots \o a_j) \o a_{j+1} \o \ldots
\o a_\ell.$$
\vskip 2pt

\begin{prop} \label{dgap}
\vi
The operation $\,\dis f_{}\vee g := {f\vdash g}\ -\ {f\dashv g}$ and Hochschild differential $\b$
give the space $C(A, A_t)_{\geq 2} := \bigoplus_{p, k \geq 2} C^p(A, \overset{\,}{A^{\o k^{^{}}}})$ the structure of an
 associative DG algebra, i.e.,
\begin{gather}
f\vee (g\vee h) = (f\vee g)\vee h,\label{assp} \\
\quad \b(f\vee g) = (\b f)\vee g + (-1)^{p-1} f\vee(\b g) \qquad \forall f \in C^p(A, A^{\o k}). \label{dgp} 
\end{gather}
\vskip 2pt

\vii The cup product $\cup$ is associative and induces the zero
map on  cohomology:
$$
H^{p}(A, A^{\o k}) \o H^{q}(A, A^{\o m}) \stackrel{_{\cup=0}}\tooo
H^{p+q}(A, A^{\o k+m-1})\quad\forall p, q\geq 1
\en\oper{whenever}\en k \geq 2 \en\oper{or}\en m \geq 2.
$$\vskip 2pt

\viii The following compatibility identities hold:
\begin{equation} \label{cbcomp}
(f \cup g)\vee h = f \cup (g\vee  h), \qquad f\vee  (g \cup h) = (f\vee g) \cup h.
\end{equation}
\end{prop}\vskip 2pt

\begin{proof}
We note the following identities (for any $x,y,z$):
\begin{gather*}
x \vdash (y \vdash z) = (x \vdash y) \vdash z, \quad x \dashv (y \dashv z) =
(x \dashv y) \dashv z, \\
x \vdash (y \dashv z) = (x \vdash y) \dashv z, \quad x \dashv (y \vdash z) = (x \dashv y) \vdash z.
\end{gather*}
The first set of identities is fairly obvious from the definition, and
the second follows because, since $y \in C^{p'}(A, A^{\o {k'}})$ for
$p', k' \geq 2$, $x$ (on the left) and $z$ (on the right) are always
applied to disjoint sets of consecutive components.  This is all we
need to prove the associativity \eqref{assp}. In fact, $\dashv$ and $\vdash$ are
mutually associative \eqref{mass}.

To prove the DG property \eqref{dgp}, we show the following two identities:
\begin{gather} \label{dvdp}
(\b f \vdash g) + (-1)^{p-1}(f \vdash \b g) = \b(f \vdash g) + (-1)^{p+1} f \cup g, \\ \label{dvdp2}
(\b f \dashv g) + (-1)^{p-1}(f \dashv \b g) = \b(f \dashv g) + (-1)^{p-1} f \cup g.
\end{gather}
Actually, in \eqref{dvdp}, we only need $m \geq 2$, and in
\eqref{dvdp2}, we only need $k \geq 2$.  These identities also prove
that the cup product is zero on cohomology. This implies part (ii),
since it is clear that the cup product is associative.

We show only \eqref{dvdp}, as the other identity is the same
verification.
Write $ m: A \o A\to A$ for the multiplication map. We compute
\begin{multline*}
((\b f \vdash g) + (-1)^{p-1}(f \vdash \b g))(a_1 \o \ldots \o a_{p+q})
= 
a_1 (f \vdash g)(a_2 \o \ldots \o a_{p+q})\\ +(-1)^{p+q+1} (f \vdash g)(a_1 \o \ldots \o a_{p+q-1}) a_{p+q} +
(-1)^{m+1} f \cup g + \sum_{i=1}^{p+q-1} (-1)^i (f \vdash g) \ccirc  m^{i,i+1}.
\end{multline*}

Finally, part (iii) of the proposition follows from Proposition \ref{kkoz}(ii) below.
\end{proof}

For each $n\geq 1$, we may introduce an operad, $As^{(n)}$, generated by $n$
binary operations, $\star_i,\ i=1,\ldots,n$, subject to the following
relations
of pairwise mutual associativity (considered also by J.-L. Loday):
\begin{equation}\label{mass}
a \star_i (b \star_j c) = (a \star_i b) \star_j c,\qquad \forall
i,j=1,\ldots,n.
\end{equation}

\begin{prop}\label{kkoz} \vi Each of the operads  $As^{(2)}$ and  $As^{(3)}$
is  Koszul and self-dual, (see \cite{GK} for a definition).

\vii
The operations $(\vee,\, \cup)$ and
$(\vdash,\, \dashv,\, \cup)$ make $C(A, A_t^+)_{\geq 2}$ an $As^{(2)}$- and an {$As^{(3)}$-algebra, respectively.}
\end{prop}
\begin{proof}[Sketch of Proof]  It is easy to see, as in the associative case, that the quadratic dual
of $As^{(2)}$ is itself and similarly for $As^{(3)}$.
The fact that these two operads are Koszul is a straightforward
consequence of Koszulity criteria due to Dotsenko and Khoroshkin
\cite{DK} and Hoffbeck \cite{Ho}.
This can also be seen directly using the same
argument as in the associative case \cite{GK}, \cite{MSS}.
Specifically, we show that the operadic homology of the
free $As^{(2)}$ or $As^{(3)}$-algebra vanishes in degrees $\geq 2$. To
this end, we split up the operadic homology complexes for  $As^{(2)}$ and
$As^{(3)}$
into direct sums of
pieces corresponding to particular sequences of operations, e.g., $(*, \star)$ would consist
of terms that multiply out to a sum of terms of the form $a * b \star c$.  Each such has the
vanishing homology property by the same proof as in the usual case
of  Hochschild homology of a free associative algebra; see \cite[\S 3.1]{L}.

Next, it is straightforward to verify the following identities:
\begin{gather*}
x \vdash (y \cup z) = (x \vdash y) \cup z, \quad x \dashv (y \cup z) = (x \dashv y) \cup z, \\
(x \cup y) \vdash z = x \cup (y \vdash z), \quad (x \cup y) \dashv z = x \cup (y \dashv z).
\end{gather*}
This yields part (ii), and also implies the identities in \eqref{cbcomp}.
\end{proof}

\subsection{Infinite order deformations}\label{defor2}
In
 the classical theory, an infinite order formal deformation of an algebra $A$
with  multiplication
map $m: A\times A\to A$ is a formally associative star-product
\begin{align}\label{class_star}
a\star a'=m(a,a')+ t\be^{(1)}(a,a')+t^2\be^{(2)}(a,a')+\ldots
\in A\lll t\rrr,\quad\be^{(k)}\in C^2(A, A),\en k\geq 1.
\end{align}

Given such a star-product, we extend the formal series $m+
t\be^{(1)}+t^2\be^{(2)}+\ldots\in\sum t^kC^2(A, A)$ $=C^2(A, A\lll
t\rrr)$, by $\k\lll t\rrr$-bilinearity, to obtain a continuous cochain
$\dis\be := \be^{(1)} + t \be^{(2)} + \ldots \in C^2_{\k\lll
  t\rrr}(A\lll t\rrr, A\lll t\rrr)$.

With an appropriate equivalence
relation on the set of associative star-products,
one has the following well-known result.

\begin{prop}\label{class_def}  Equivalence classes of  associative
star products \eqref{class_star}
are in one-to-one  correspondence
with gauge equivalence classes  in the set of solutions of
the following Maurer-Cartan equation
$$
\{\be\in C^2_{\k\lll t\rrr}(A\lll t\rrr, A\lll t\rrr)\;\mid\;
 \b_{A\lll t\rrr}(\be)+(1/2)t\{\be,\be\}_{A\lll t\rrr}=0\}.\eqno\Box
$$
\end{prop}
\medskip

To consider free product deformations of an algebra $A$, let
$\wh{A}_t:=\prod_{k\geq 0}\, A^{2k}_t\cong \prod_{m \geq 1} A^{\o m}$
be the  completion of the  free product algebra $A_t$
in the $t$-adic topology, and write $\wh{A}^+_t:=\prod_{k\geq 1}\,
A^{2k}_t$
for the corresponding augmentation ideal.

An infinite order free product deformation of $A$ is, by definition, a formally
associative star-product of the form
\begin{equation}\label{star}
a\starb a'= aa'+ \be^{(1)}(a,a')+\be^{(2)}(a,a')+\ldots,\qquad
\be^{(k)}\in C^2(A, A^{2k}_t).
\end{equation}

In more detail, given an arbitrary sequence $\be^{(k)} \in C^2(A, A^{2k}_t),\
k= 1,2,\ldots$, of two-cochains,
we first extend each map $\be^{(k)}$
to a $\k[t]$-bilinear  map $\be^{(k)}: {A}_t\times {A}_t\to
{A}_t$
given by the formula
$$
\be^{(k)}:\
(a_1 t a_2 t \ldots t a_m) \times (b_1 t \ldots t b_n) \lto
 a_1 t \ldots t a_{m-1} t \be(a_m \o b_1) t b_2 t \ldots t b_n.
$$

For any $u,u'\in A_t$,
the corresponding
 formal series $uu'+ \be^{(1)}_t(u,u')+\be^{(2)}_t(u,u')+\ldots$
clearly converges in $\wh{A}_t$.
In this way, we obtain a well defined and continuous 
 $\k[t]$-bilinear map $A_t\times A_t\to \wh{A}_t$,
that  can be uniquely extended,
by continuity, to a map
$\be:\wh{A}_t\times\wh{A}_t\to\wh{A}_t$.
We are interested in  those star-products
\eqref{star} which give rise to an
{\em associative} product $\beta$,  on $\wt A_t$.

We define a natural equivalence relation on such  star-products
as follows. First, for any map $\phi: A\to A_t$ we define
a map $\phi_t: A_t\to A_t$ by a 
Leibniz-type formula:
\begin{equation}\label{f_t}
\phi_t:\
a_1\,t\,a_2\,t\ldots t\,a_n\lto
\sum\nolimits_{k=1}^n a_1\,t\ldots a_{k-1}\,t\,\phi(a_k)\,t\,
a_{k+1}\,t\ldots t\,a_n,
\end{equation}
 for any $a_1,\ldots,a_n\in A$.
Thus, given a sequence of one-cochains
$f^{(k)}\in C^1(A, A_t^{2k}),\,k=1,2,\ldots$,  we
have a continuous map
$f= \Id+f^{(1)}_t+f^{(2)}_t+\ldots: \wh{A}_t\to\wh{A}_t$. Furthermore,
given any star-product $\starb$, we define a new star product by the
formula $\dis a\star_{_\gamma}a':=f\inv(f(a)\starb f(a')).  $ We say
that the star-products $\star_{_\gamma}$ and $\starb$ are {\em
  equivalent}.

Given a star-product \eqref{star},
we form
$\be:=\be^{(1)}+\be^{(2)}+\ldots\in C^2(A, \wh{A}^+_t)$,
an associated  two-cochain.
We further define gauge equivalence of chains 
to be infinitesimally generated by the following $C^1(A,
\wh{A}_t^+)$-action:
$\phi:\ \beta\mapsto\phi \cdot \beta$, where
$$
\phi \cdot \beta(a_1 \o a_2) = \phi(a_1) a_2 + a_1 \phi(a_2) - \phi(a_1
a_2) 
+ \beta'(\phi(a_1) \o a_2) + \beta'(a_1 \o \phi(a_2)) - \phi_t  \ccirc \beta'(a_1 \o a_2).
$$
Here, $
\phi_t$ is defined according to formula \eqref{f_t}, and we put
$$
\beta'(a_1 t \ldots t a_m \o c_1 t \ldots t c_n) = a_1 t \ldots t a_{m-1} t
(\beta(a_m \o c_1)) t c_2 t \ldots t c_n.
$$
 
The following result provides a cohomological description of free product deformations,
similar to the one given in
Proposition \ref{class_def}
(below,
$\b$ stands for the Hochschild
differential).

\begin{thm}
\vi Linear maps $\be^{(i)}$ in \eqref{star} define an associative product
 on $\wh{A}_t$
   if and only if   the Maurer-Cartan equation,
\begin{equation} \label{mce}
\b(\beta) + \frac{1}{2} \ \beta\vee\beta = 0,\qquad
\oper{holds\en for}\quad\beta := \beta^{(1)}+\beta^{(2)}+\ldots\in
C^2(A, \wh{A}_t^+).
\end{equation}

\vii Star products are equivalent  if and only if   the solutions
of the  equation \eqref{mce} are gauge-equivalent.
\end{thm}

\begin{proof} In this proof (only) we temporarily change our notation
and given $\beta^{(m)}\in C^2(A, A^{\o m+1})$, write
$\beta_m^{12}\in C^3(A, A^{\o m+2})$ for the map
 $a \o b \o c \mto \beta^{(m)}(a,b)\o c$, etc.

It is easy to see that it suffices to check associativity on $A^{\o 3}$,
namely that $(a \star b) \star c = a \star(b \star c)$ for all $a, b, c
\in A$.  This is equivalent to the following (similar to
\eqref{firstass}), for all $p \geq 2$:
\begin{multline}
\be^{(p)}(a \o bc) + a\be^{(p)}(b \o c) + \sum_{m+n=p} \be_m^{12} \ccirc \be_n^{23} (a \o b \o c) \\ = \be^{(p)}(ab \o c) + \be^{(p)}(a \o b) c + \sum_{m + n = p} \be_n^{m,m+1} \ccirc \be_m^{12} (a \o b \o c).
\end{multline}
which is the Maurer-Cartan equation \eqref{mce}.  Part (ii) is then not difficult to verify.
\end{proof}

Below, we summarize a few basic  properties of free product deformations which
are entirely analogous
to the well-known properties of ordinary one-parameter formal deformations:

\begin{itemize}
\item[(1)] First order free product deformations are classified by
  $H^2(A, A \o A)$.

\item[(2)] The obstruction to extending a first-order
  deformation to second order lives in
  $H^3(A, A^{\o 3})$.

\item[(3)] Let $\star_n$ be an associative product on $A / (A_t^+)^{n+1}$ of the form
\begin{equation}
a \star_n b = ab + \sum_{m=1}^n  \be^{(m)}(a \o b),\qquad \be^{(m)}\in
C^2(A, A^{m+1})=C^2(A, A^{2m}_t).
\end{equation}
There is an obstruction  in $H^3(A, A^{\o  n+1})$
to the existence of  $\be^{(n+1)} \in C^2(A, A^{\o (n+2)})$,
such that the formula $a \star_{n+1} b:=a \star_n b +  \be^{(n+1)}(a \o b)$ gives an associative
product on $A / \ldp t \rdp^{n+2}$.

Explicitly, the condition on 
$\be^{(n+1)}$ reads
\begin{equation} \label{mcex}
\b\be^{(n+1)} = \sum_{i+j-1 = n+1}\ \be^{(i)}\vee \be^{(j)}.
\end{equation}

\item[(4)] If the obstruction in (3) vanishes then,
 the space of possible $\be^{(n+1)}$ (up to equivalence of the
resulting star product, $\star + \be^{(n+1)}$ modulo
$\ldp t \rdp^{n+2}$), is
$H^2(A, A^{\o(n+2)})$.
\end{itemize}

\begin{proof}
In  degrees $n=1,2$, the Maurer-Cartan equation  
\eqref{mcex} says that
$$
\b \be^{(1)} = 0, \qquad\text{and}\qquad
\b \be^{(2)} = \be^{(1)}\vee\be^{(1)}.
$$
Using
\eqref{dgp}, we find $\b (\be^{(1)}\vee \be^{(1)}) = (\b \be^{(1)})\vee \be^{(1)}
+ \be^{(1)}\vee (\b \be^{(1)})
 = 0 + 0 = 0$.
This yields (1)--(2).

In general, if $\be^{(1)}, \ldots, \be^{(n)}$ satisfy the Maurer-Cartan
conditions up to $O(t^{n+1})$ (i.e., $\b \be^{(m)} = \sum_{i+j-1 = m} \
\be^{(i)}\vee\be^{(j)}$
 for $m \leq n$), then we consider $\b$ of the RHS of \eqref{mcex}:
\begin{multline*}
\sum_{i+j-1=n+1} \b(\be^{(i)}\vee \be^{(j)})= \sum_{i+j-1=n+1} 
\big[(\b \be^{(i)})\vee
\be^{(j)} - \be^{(i)}\vee (\b \be^{(j)})\big] \\
= \sum_{i+j+k-2=n+1} \big[(\be^{(i)}\vee \be^{(j)})\vee \be^{(k)} -
\be^{(i)}\vee (\be^{(j)}\vee \be^{(k)})\big] = 0,
\end{multline*}
where  we used both \eqref{dgp} and \eqref{assp}.  Thus, the RHS is
indeed a Hochschild three-cocycle.  Thus, if this represents the zero
class of $H^3(A, A^{\o  n+1})$ (i.e., it is a Hochschild
three-coboundary), then the space of choices of $\be^{(n+1)}$ is the space
of Hochschild two-cocycles.  
Furthermore, we have the freedom of conjugating by
automorphisms $\phi: A \rightarrow A$ of the form $\phi = \Id + \phi'$
as follows:
$$
\phi^{-1}(\phi(a) \star \phi(b) ) \equiv a \star b + \phi'(a) b + a \phi'(b) - \phi'(ab) \pmod {\ldp t \rdp^{n+2}}.
$$

We conclude that  the space of
$\be^{(n+1)}$'s, taken up to equivalence of the obtained
star product on $A / \ldp t \rdp^{n+2}$, is $H^2(A, A^{\o n+2})$.
\end{proof}
\subsection{Deformations of NCCI algebras} 
It will be convenient below to work in a slightly more general setting
of  deformations that are not necessarily written in the form of a star
product.

To define such deformations, fix an augmented (not necessarily
commutative) associative algebra $R$, and let $R^+\sset R$ be the
augmentation ideal.  Given an algebra $\wt A$ and an algebra embedding
$R\into \wt A$, write $\ldp R^+\rdp\sset \wt A$ for the two-sided
ideal in $\wt A$ generated by $R^+$.  We will view $R$ and $\wt A$ as
filtered algebras equipped with the $R^+$-adic and $\ldp R^+\rdp$-adic
descending filtrations, respectively, and let $\gr R$ and
$\gr \wt A$ denote the associated graded algebras. Thus, there is a
canonical algebra map $\gr R\to\gr \wt A$.

Given an algebra $A$ and an algebra isomorphism
$\phi:\ \wt A/\ldp R^+\rdp\iso A$,
we say that $\wt A$ is a deformation of $A$ over $R$. 
We may view $A$ as a graded algebra concentrated in degree zero.

\begin{defn} \label{ncdd} The deformation  $\wt A$ of $A$ over $R$
is said to be a {\sf flat} {\em free product formal deformation}
if the algebra  $\wt A$ is complete in the $\ldp R^+\rdp$-adic topology,
and the   maps $\phi\inv: A\to \wt A/\ldp R^+\rdp$
and $\gr R\to\gr \wt A$ induce a graded
algebra isomorphism
\begin{equation}\label{flatdef}
A*\gr R\iso\gr \wt A.
\end{equation}
\end{defn}

Now, fix  $V$, a $\Z_+$-graded
finite dimensional vector space, and let
 $F := TV$. Let  $L \subset T V$ be
a  finite-dimensional vector subspace.
 Assuming certain favorable conditions, we can describe the equivalence
classes of 
{\em all}  infinite order 
 free product deformations of  an algebra of the form $A=F / \ldp L \rdp$
quite explicitly. 

To explain this, 
write $\wh{F}_{t}$ for
 the standard completion of the algebra
$F_t=(TV)*\k[t]$ and $\wh{F}_t^+\sset \wh{F}_t$ for
 the augmentation ideal. Given  any linear
map $\phi: L \to \wh{F}_t^+$, we introduce a $\k\lll t\rrr$-algebra
\begin{equation}\label{defa}
A_\phi := \wh{F}_t / \ldp x - \phi(x)\rdp_{x \in L}.
\end{equation}
It is clear that the projection $\wh{F}_t\onto\wh{F}_t/\wh{F}_t^+=F$
induces an algebra isomorphism    $A_\phi / \ldp t \rdp\iso A$. 

Thus, we may view the algebra $A_\phi$ as a
one-parameter infinite order 
 free product deformation of $A$. This deformation is {\bf not} necessarily
 flat,
in general, i.e., the corresponding map \eqref{flatdef} for
$\wh{A}=A_\phi$  may fail to be an 
isomorphism.

To formulate a sufficient condition for flatness, we recall the notion
of a {\em noncommutative complete intersection} (NCCI); see \cite{EG}.
An algebra of the form $A= TV/ \ldp L \rdp$ is said to be an NCCI if
the two-sided ideal $J:=\ldp L \rdp$ has the property that $J/J^2$ is
{\em projective} as an $A\o A^{op}$-module.  Moreover, such a linear
subspace $L \subset TV$ is called \emph{minimal} if $L \cap J^2 = 0$,
or equivalently, $L$ has minimal dimension (assuming $J$ is
finitely-generated).

An NCCI algebra $A$ is known to have Hochschild dimension
$\leq 2$,  so that
$H^3(A, A \o A) = 0$ \cite{EG}. Thus, free product deformations of
$A$ 
are unobstructed.  
Moreover,  we have the following.

\begin{prop}\label{propdef}  Let  $A= TV/ \ldp L \rdp$ be an NCCI, with
$L \subset TV$ minimal, and let
 $\phi: L \to \wh{F}_t^+$ be a linear map. Then, we have the following.

\vi The deformation  $A_\phi$ defined in \eqref{defa} is flat.

\vii Any flat one-parameter infinite order 
 free product deformation of $A$ is equivalent to a deformation of
the form $A_\phi$ for an appropriate map $\phi$.

\viii Two deformations $A_\phi$ and $A_\psi$ associated,
respectively, to linear
  maps $\phi,\,\psi\in\Hom_\k(L, \wh{F}_t^+)$, are equivalent  if and only if   there
exists a linear map $f: V\to A_t^+$, such that
\begin{equation}\label{equivdef}
\pi \ccirc (\phi - \psi) = \Theta_f |_{L}.
\end{equation}
\end{prop}

In the last formula, we used the notation
$\pi: \wh{F}_t \onto A_t$
for the canonical quotient map and, given a linear map
 $f: V\to A_t^+$, write
$$\Theta_f(v_1 v_2 \cdots
v_n) := \sum_{i=1}^n \pi(v_1 v_2 \cdots v_{i-1}) f(v_i)
\pi(v_{i+1} \cdots v_n),\quad \forall v_1, \ldots, v_n \in V. 
$$
\begin{rem}
Our proof below shows that, in the case where the image of the map $\phi$ is contained in
the subalgebra $F_t\sset\wh{F}_t$, we may replace the algebra
$A_\phi$, in \eqref{defa}, by
$F_t / \ldp x - \phi(x)\rdp_{x \in L}$,  its non-completed
counterpart. In this way, we obtain a {\em genuine}, rather than merely
a
`formal',
 flat free product deformation 
of $A$.
\end{rem}

\begin{proof}[Proof of Proposition \ref{propdef}] \vi By the inductive
  argument of \cite[Proposition 4.2.1]{S3}, one may show that the NCCI
  property is equivalent to the statement that the canonical
  surjection $\text{gr}_{\ldp L \rdp} F \onto A * TL$ is an
  isomorphism.  Thus, we conclude that the surjection
  $A_t \onto \text{gr}_{\ldp t \rdp} A_\phi$ is an isomorphism.
   It follows that $A_\phi$ is a free product deformation of $A$ over $\k\lll t \rrr$.

\vii For an NCCI algebra, there is a standard Anick's free
resolution of $A$ as an $A$-bimodule \cite{An2}:
\begin{equation}\label{anres}
0 \rightarrow A \o L \o A \rightarrow A \o V \o A \rightarrow A \o A \onto A \rightarrow 0,
\end{equation}
where the first map is the restriction to $L$ of the map
$$a \o (v_1 v_2 \cdots v_n) \o b \lto \sum_{i=1}^n a v_1 \cdot\ldots\cdot
v_{i-1} \o v_i \o v_{i+1} \cdots v_n b,\quad v_i \in V,$$ and the second map has the form
$a \o v \o b \mapsto av \o b - a \o vb$.

One may use  Anick's
resolution to compute Hochschild cohomology. We see in particular that
the group  $H^2(A, A \o A)$ is a quotient of $\Hom(L,  A \o A)$.
Now let  $A_\phi$ be the deformation associated
with $\phi\in\Hom(L, \wh{F}_t^+)$.
Then, it is easy to check that the element of $H^2(A, A \o A)$
corresponding to the induced first-order deformation $A_\phi/ \ldp t \rdp^2$
is represented by the composite  
$$L\to\wh{F}_t^+\onto \wh{F}_t^+/(\wh{F}_t^+)^2=F\o F\onto A \o A,$$
of the map $\phi$ 
followed by two
natural projections (cf.~\cite[Lemma 10.2.1]{CBEG}).

Furthermore, by the inductive description of all possible star-products in \S
\ref{defor2}, the deformations $A_\phi$ must exhaust all possible
deformations (note that all possible classes of $H^2(A, A^{\o m})$ at
every step of the way are attained, which is as it must be, since
$H^3(A, A \o A) = 0$).

\viii At the first-order stage, we see from Anick's resolution
\eqref{anres} that two elements of the space $\Hom_k(L, A \o A)$ yield
the same cohomology class in $H^2(A, A \o A)$  if and only if   they
differ by $\Theta_f$ from condition \eqref{equivdef} modulo
$\ldp t \rdp^2$, for some $f$.  
The desired result now follows from  the inductive construction of all free-product
deformations given in  \S \ref{defor2}.
\end{proof}

\begin{rem}\label{final} \vi In general, given an arbitrary
algebra $A$ such that $H^3(A,A\o A)=0$,
one can show that there exist `versal'  free product deformations of $A$.
The base  of such a versal deformation is a 
completed tensor algebra of the vector space $H^*$, where
$H:=H^2(A,A\o A)$.

\vii Proposition \ref{propdef} may be generalized easily to the case
where the ground field $\k$ is replaced by a ground ring $R$, a finite dimensional 
semisimple $\k$-algebra, as in \cite{EG}.
Such a generalized version of  Proposition \ref{propdef}
applies to preprojective algebras
of non-Dynkin quivers, in particular. Thus the proposition 
may be viewed as
a generalization of \cite{CBEG}, Theorem 10.1.3.

\viii Let $\pi_1(X)$ be the fundamental group of a compact oriented
Riemann surface $X$ of genus $\geq 1$.  The group algebra
$\k[\pi_1(X)]$ may be thought of as a multiplicative analogue of the
preprojective algebra of a non-Dynkin quiver (being
non-Dynkin corresponds to the condition that the Euler characteristic
of $X$ be nonpositive).  Accordingly, there is a similar construction
of free product deformations of the group algebra as follows.

Let $g$ be the genus of $X$, and write
  $a_1, \ldots, a_g, b_1, \ldots, b_g$ for the standard loops around the
  handles, which generate $\pi_1(X)$.
The group $\pi_1(X)$ is the quotient
of $\Gamma$, the free group generated
  by the letters $a_i, b_i$, by the normal subgroup generated 
by the following element:
$$\ggamma:=
(a_1 b_1 a_1^{-1} b_1^{-1})\cdot\ldots\cdot
 (a_g b_g a_g^{-1} b_g^{-1})\
  \in \ \Gamma.$$

To construct free product deformations of the group
algebra $\k[\pi_1(X)]$, we put $F := \k[\Gamma]$.  The algebra $F$
is a `multiplicative analogue' of a free algebra.
To any element $u \in 1 + F_t^+$,
we associate an algebra
$\dis A_u := F_t / \ldp \ggamma - u \rdp$.

There is a `multiplicative analogue' of Proposition \ref{propdef},
saying that the algebra
$A_u$ gives a flat free-product deformation of the group
algebra $\k[\pi_1(X)]$, and moreover that these are all such
deformations up to equivalence. One may prove this result
by using the fact that
the prounipotent completion of $\k[\pi_1(X)]$ is isomorphic to
 a completion of an
algebra of the form
$\dis\k\langle x_1,\ldots,x_g,y_1,\ldots,y_g\rangle/
\ldp[x_1,y_1]+\ldots+[x_g,y_g]\rdp$.

This example may be generalized to the situation of orbifold surfaces of
nonpositive Euler characteristic (the latter also yields NCCI algebras).
\end{rem}

\appendix


\newcommand{\chr}{\operatorname{char}}
\newcommand{\supp}{\operatorname{Supp}}
\renewcommand{\Ker}{\operatorname{Ker}}
\renewcommand{\im}{\operatorname{Im}}
\newcommand{\codim}{\operatorname{codim}}
\newcommand{\Diff}{\operatorname{Diff}}
\newcommand{\Eu}{\operatorname{Eu}}
\renewcommand{\eu}{\operatorname{eu}}
\renewcommand{\D}{\operatorname{D}}
\renewcommand{\R}{\mathbf{R}}
\renewcommand{\Hom}{\operatorname{Hom}}
\renewcommand{\End}{\operatorname{End}}
\newcommand{\Isom}{\operatorname{Isom}}
\newcommand{\isomo}{\overset{\sim}{=}}
\newcommand{\isomoto}{\overset{\sim}{\to}}
\renewcommand{\id}{{\mathtt{Id}}}
\renewcommand{\Tr}{\operatorname{Tr}}
\newcommand{\shHom}{\underline{\operatorname{Hom}}}
\newcommand{\shEnd}{\underline{\operatorname{End}}}
\newcommand{\shExt}{\underline{\operatorname{Ext}}}
\newcommand{\shAut}{{\underline{\operatorname{Aut}}}}
\newcommand{\shIsom}{{\underline{\operatorname{Isom}}}}
\newcommand{\shEq}{{\underline{\operatorname{Eq}}}}
\newcommand{\Eq}{\operatorname{Eq}}

\renewcommand{\Ext}{\operatorname{Ext}}
\newcommand{\Pic}{\operatorname{Pic}}
\renewcommand{\ch}{\operatorname{ch}}
\renewcommand{\rk}{\operatorname{rk}}
\newcommand{\Symm}{\operatorname{Sym}}
\newcommand{\cl}{\operatorname{cl}}
\renewcommand{\pr}{{\mathtt{pr}}}
\renewcommand{\ev}{\operatorname{ev}}

\newcommand{\EXT}{{\mathcal E\!\scriptstyle{XT}}}
\renewcommand{\DR}{\mathtt{DR}}
\newcommand{\de}{\displaystyle\frac{\partial\ }{\partial\epsilon}}
\newcommand{\vac}{{\mathbf{1}}}
\newcommand\dual[1]{{#1}^{\vee}}
\renewcommand{\ip}{{\langle\ ,\ \rangle}}
\newcommand{\Atiyah}{\alpha}
\newcommand{\pont}{{\scriptstyle{\Pi}}}
\newcommand{\Gg}{{\mathfrak g}}
\newcommand{\coComm}{\operatorname{coComm}}
\renewcommand{\Ad}{\operatorname{Ad}}
\newcommand{\one}{{1\!\!1}}
\newcommand{\DQ}{\operatorname{DQ}}
\newcommand{\shDQ}{\underline{\operatorname{DQ}}}
\renewcommand{\Triv}{\operatorname{Triv}}
\newcommand{\MC}{\operatorname{MC}}
\newcommand{\shMC}{\underline{\operatorname{MC}}}
\newcommand{\Def}{\operatorname{Def}}
\newcommand{\shDef}{\underline{\operatorname{Def}}}
\newcommand{\SP}{\operatorname{SP}}
\newcommand{\Hol}{\operatorname{Hol}}
\renewcommand{\Aut}{\operatorname{Aut}}
\newcommand{\Tw}{\operatorname{Tw}}
\newcommand{\TW}{{\mathtt{Tw}}}
\newcommand{\curv}{{\mathfrak\scriptstyle C}}
\renewcommand{\Mat}{{\mathtt{Mat}}}
\newcommand{\Tot}{{\mathtt{Tot}}}
\renewcommand{\Der}{{\mathtt{Der}}}
\newcommand{\mmu}{{\mu\!\!\!\mu}}
\newcommand{\SIGMA}{{\scriptstyle\Sigma}}
\newcommand{\Alg}{\operatorname{ALG}}
\newcommand\real[1]{{\vert{#1}\vert}}
\newcommand{\cat}{\mathbf{Cat}}
\newcommand{\Stack}{\operatorname{Stack}}

\newcommand{\Top}{\mathtt{Top}}
\newcommand{\Cov}{\mathtt{Cov}}
\newcommand{\dbar}{\overline{\partial}}
\newcommand{\AuxData}{\mathtt{AD}}
\newcommand{\liminv}{\mathop{\varprojlim}\limits}
\renewcommand\s[1]{{(#1)}}
\newcommand{\stack}{\mathtt{st}}
\renewcommand{\A}{\operatorname{\mathcal A}}
\renewcommand{\LL}{\operatorname{\mathcal L}}
\newcommand{\cma}{\mathtt{cma}}
\newcommand{\conv}{\mathtt{conv}}
\renewcommand{\triv}{\mathtt{triv}}
\newcommand{\mat}{\mathtt{mat}}
\newcommand{\dd}{\mathtt{dd}}
\renewcommand{\op}{\mathtt{op}}
\newcommand{\dR}{\widehat{\mathrm{d}}_\mathrm{d\!\! R}}
\renewcommand{\bc}{\mathtt{B}}
\newcommand{\jet}{{{}^{(\infty)}\!\!\mathcal{J}}}
\newcommand{\conn}{{{}^{(\infty)}\!\nabla}}


\makeatletter
\renewcommand{\subsubsection}{\@startsection
{subsubsection}%
{2}%
{0mm}%
{-\baselineskip}%
{-0.5\baselineskip}%
{\normalfont\normalsize\bfseries }}%
\makeatother

\section{On the morphism from periodic cyclic homology to equivariant cohomology, by Boris Tsygan} \label{s:Morphism}
\subsection{Motivation} \label{s:Motivation} Since Hochschild chains are a noncommutative analogue of forms, it is natural to ask what algebraic structure that is carried by forms can be defined for chains of an associative algebra. The structure we will discuss involves not only forms but also vector fields (their noncommutative analogue is the Lie algebra of derivations) and functions (their analogue is the algebra itself). It is well known that, for a vector field $X$, two operators on forms are defined: the contraction $\iota _X$ and the Lie derivative $L_X$. We will write all the formulas for a graded commutative ring, since this will give us an intuition about the signs. For a graded derivation $X$ of such an algebra ${\mathcal A}$, let $\iota_X$ be the derivation of $\Omega^*_{{\mathcal A}/k}$ of degree ${\rm{deg}}(X)-1$ sending $a$ to zero and $da$ to $X(a)$ for $a\in {\mathcal A}.$ Let $L_X$ be the derivation of degree ${\rm{deg}}(X)$ sending $a$ to $X(a)$ and $da$ to $(-1)^{{\rm{deg}}(X)}dX(a).$
Put $|X|={\rm{deg}}(X)+1.$
Together with the de Rham differential $d$, the above operations are subject to Cartan relations
\begin{equation}\label{eq:Cartan 1}
[L_X, L_Y]=L_{[X,Y]};\; [L_X, \iota_Y]=\iota_{[X,Y]};\;
\end{equation}
$$[\iota_X, \iota_Y]=0;\;[d, L_X]=0;\;[d, \iota_X]=(-1)^{|X|-1}L_X$$
All commutators are taken in the graded sense. As a corollary, we get the classical formula for the de Rham differential: for an $n$-form $\omega$,
one has
\begin{equation}\label{eq:Chevalley 1}
\iota_{X_1}\ldots \iota_{X_{n+1}}d\omega=
\end{equation}
$$(\sum_{j=1}^{n+1}(-1)^{K_{j}}L_{X_j}\iota_{X_1}\ldots {\widehat{\iota_{X_j}}}\ldots \iota_{X_{n+1}}+
$$
$$\sum_{i<j}(-1)^{M_{ij}}\iota_{[X_i,X_j]}\ldots {\widehat{\iota_{X_j}}}\ldots{\widehat{\iota_{X_j}}}\ldots )\omega $$
where
\begin{equation}\label{eq:sign1}
K_j={\sum_{a}|X_a|+ |X_j|(\sum_{b<j}|X_b|+1)}
\end{equation}
and
\begin{equation}\label{eq:sign2}
M_{ij}=\sum_{a}|X_a|+ |X_j|\sum_{b<j, b\neq i}|X_b|+|X_i|(\sum_{c<i}|X_c|+1)+1.
\end{equation}
This follows from \eqref{eq:Cartan 1}.

 Now, for an element $f$ of ${\mathcal A}$, put $|f|={\rm{deg}}(f),$  $\iota _f\omega=f\omega,$ and $L_f\omega=(-1)^{|f|-1}df\wedge\omega$. In addition to the above equations, the following are true:
\begin{equation}\label{eq:Cartan 3}
[L_X, L_f]=L_{X(f)};\; [L_X, \iota_f]=\iota_{X(f)};\;[\iota_f, \iota_g]=0;\;[d,\iota_f ]=(-1)^{|f|-1}L_f.
\end{equation}

We can combine the two sets of equations \eqref{eq:Cartan 1} and \eqref{eq:Cartan 3}, into one set that looks same as \eqref{eq:Cartan 1}, but
$X$ and $Y$ are now elements of the cross product of the graded Lie algebra ${\rm{Der}}({\mathcal A})$ by the Abelian graded Lie algebra ${\mathcal A}[1].$

 The  algebraic structure described above, and a much stronger additional structure, generalizes to the noncommutative setting. These generalizations become progressively more and more difficult and inexplicit. However, the very first level of noncommutative calculus, namely formula \eqref{eq:Chevalley 1} (and its slight generalization to the case when $X_i$ are vector fields or functions) turns out to be as easy as one can expect.
\subsection{The map $\chi$}\label{ss:the map chi}
Let ${\mathcal A}$ be a graded associative algebra. Let $\A[1]\oplus \Der (\A)$ be the differential graded Lie algebra defined as follows: $\A[1]$ is Abelian, $\Der (\A)$ is a subalgebra whose adjoint action on $\A[1]$ is the natural one, the differential $\delta$ sends $a\in \A[1]$ to $(-1)^{|a|}\ad (a)$ and is zero on $\Der(\A)$. Let $\LL$ be a DG Lie subalgebra of $\A[1]\oplus \Der (\A)$. Let $K$ be a graded space on which $\LL$ acts so that elements of $\A[1]$ act by zero. Let $\tau: \A\to K$ be an $\LL$-equivariant trace. We extend it by zero to the entire Hochschild complex $C_{-\bullet}(\A).$

Define for $a\in \A[1]$
\begin{equation}\label{eq: iota hoch 1}
\iota_a(a_0\otimes a_1 \otimes \ldots \otimes a_n)=(-1)^{|a||a_0|}a_0a\otimes a_1 \otimes \ldots \otimes a_n
\end{equation}
and for $D\in \Der(\A)$
\begin{equation}\label{eq: iota hoch 2}
\iota_D(a_0\otimes a_1 \otimes \ldots \otimes a_n)=(-1)^{|D||a_0|}a_0D(a_1)\otimes a_2 \otimes \ldots \otimes a_n
\end{equation}
Define, for $X_1, \ldots, X_n\in \LL$, and for a Hochschild chain $c$,
\begin{equation}\label{eq:chi}
\chi (X_1, \ldots, X_n)(c)=\sum _{\sigma\in S_n}\pm \tau (\iota_{X_{\sigma(1)}}\ldots \iota_{X_{\sigma(n)}}c);
\end{equation}
the sign is computed as follows: a permutation of $X_i$ and $X_j$ introduces a sign $(-1)^{|X_i||X_j|}.$ (Note that $|X|$ is the degree of the operator $\iota _X.$) It turns out that $\chi$ defines a cocycle of the complex $$C^\bullet (\LL, \Hom (C_{-\bullet}(A), K)[[u]])$$ with the differential $b+uB+\delta+u\partial _{\operatorname{Lie}}$; the action of $L$ on $\Hom (C_{-\bullet}, K)$ is induced by the action on $K$. In other words,
\begin{prop}\label{prop:chevalley}
$$\chi (X_1, \ldots, X_n)((b+uB)(c))=\frac{1}{n!}(\sum (-1)^{L_i} \chi (X_1, \ldots, \delta X_i, \ldots, X_n)+$$
$$u\sum (-1)^{M_{ij}} \chi ([X_i,X_j], \ldots, \widehat{ X_i}, \ldots,\widehat{ X_j},\ldots, X_n)+$$
$$u\sum (-1)^{K_j} X_i \chi (X_1 \ldots, \widehat{ X_i}, \ldots, X_n))(c)$$
\end{prop}
(cf. \cite{NT}, \cite{NT1}). Here  $L_i=\sum_{a\geq i}|X_a|;$ $K_j$ is as in \eqref{eq:sign1} and $M_{ij}$ as in \eqref{eq:sign2}.
\begin{proof} We observe directly that
$[b, \iota_X]=\iota_{\delta X}.$ This implies immediately that the formula in the statement of the proposition is true modulo $u$, since $\tau$ is zero on the image of $b$. We denote $\chi (X_1, \ldots, X_n)$ by $\chi (X_1 \ldots X_n)$. Let $c=a_0\otimes\ldots\otimes a_n.$
Let $D$ be an odd derivation and $x$ an even element of ${\mathcal A}$. We are going to prove Proposition \ref{prop:chevalley} first in the case when all $X_i$ are equal to $D$ or $x$, so $|X_i|=0$ and the signs in the statement of the proposition are $K_i=L_j=+1;$ $M_{ij}=-1$. By definition,
$$\chi(D^{n}x^N)=\frac{n!N!}{(N+n)!}\sum _{N_0+N_1+\ldots+N_n=N}\tau (a_0 x^{N_0}D(a_1)x^{N_1}\ldots D(a_n) x^{N_n})$$
therefore
$$\chi(D^{n+1}x^N)(Bc)=\frac{(n+1)!N!}{(N+n+1)!}\times$$
$$\times\sum _{\sum_0^{n+1} N_k=N}\sum_{i=0}^n \pm\tau (x^{N_0}D(a_{i+1}) x^{N_1}D(a_{i+2})x^{N_2}\ldots D(a_{i}) x^{N_{n+1}});$$
the sign in the above is $(-1)^{\sum_{k>i}(|a_k|+1)\sum_{k\leq i}(|a_k|+1)}.$ 
Using the fact that $\tau$ is a trace, re-indexing the $N_k$, and re-ordering the $D(a_i)$ (which cancels out the previous sign), we see that the above is equal to
$$\frac{(n+1)!N!}{(N+n+1)!}\sum _{\sum_0^{n} N_k=N}\sum_{i=0}^n\sum_{N_i'+N_i''=N_i}\tau(D(a_0)x^{N_0}\ldots D(a_n)x^{N_n})=$$
$$(N+n+1)\frac{(n+1)!N!}{(N+n+1)!}\sum _{\sum_0^{n} N_k=N} \tau(D(a_0)x^{N_0}\ldots D(a_n)x^{N_n});$$
by the Leibniz identity for $D$, we rewrite this as
$$\frac{(n+1)!N!}{(N+n)!}(\sum _{\sum_0^{n} N_k=N}  D\tau(a_0x^{N_0}\ldots D(a_n)x^{N_n})$$
$$-\frac{1}{2}\sum _{\sum_0^{n} N_k=N} \sum_{i=1}^n \pm \tau(a_0x^{N_0}\ldots [D,D](a_i)x^{N_i}\ldots D(a_n)x^{N_n})$$
$$-\sum _{\sum_0^{n} N_k=N} \sum_{i=0}^n \sum_{N_i'+N_i''=N_i}\pm \tau(a_0x^{N_0}\ldots x^{N_i'}D(x)x^{N_i''}\ldots D(a_n)x^{N_n}));$$
the signs appearing in the above terms from the Leibniz identity, with an overall minus sign, are precisely the signs from the definition of $\iota_{[D,D]}$ and $\iota_{D(x)}.$ We see that the above is equal to
$$(n+1)D\chi(D^{n}x^N) -\frac{n(n+1)}{2}\chi([D,D]D^{n-1}x^N)-$$
$$-(n+1)N\chi(D^nD(x)x^{N-1}).$$
Therefore
$$\chi(D^{n+1}x^N)( Bc)=(n+1)D\chi(D^{n}x^N) (c)-$$
$$-\frac{n(n+1)}{2}\chi([D,D]D^{n-1}x^N)(c)-(n+1)N\chi(D^nD(x)x^{N-1})(c).$$
This is exactly the equality of the terms of the formula in the statement of the proposition that contain $u$, in the special cases $X_1=\ldots =X_{n+1}=D$ and $X_{n+2}=\ldots =X_{N+n+1}=x$ (and with $n$ replaced by $N+n$). To prove the general case, tensor our algebra by $k[t_1, \ldots, t_n]$ where $|t_i|=-|X_i|,$ put $X=t_1X_1+\ldots t_n X_n,$ apply the special case to $\chi(X,\ldots, X),$ and look at the coefficient at $t_1\ldots t_n.$
\end{proof}
\subsection{Construction of the morphism}\label{ss:construction of the morphism}
Now let $\A=\Omega (\Rep (A))\otimes \End(V)$, $\Rep (A)$ being the scheme of representations of an algebra $A$ in a given finite dimensional space $V$. Consider the subalgebra $\End(V)$ consisting of constant functions. Let $K=\Omega (\Rep(A))$ and $\tau: \A\to K$ be the matrix trace. Let $\Gg=\End(V)$ viewed as a Lie algebra.
\begin{prop}\label{prop:map cc to eq(rep)} The map
$$c, x\to \tau (\exp(\iota_{d+x}))(c),$$
$x\in \Gg$ and $c\in C_{-\bullet}(\A),$
composed with the morphism induced by ${\operatorname {ev}}: A\to \A$, defines a morphism of complexes
$${\rm{HKR}}(x): C_{-\bullet}(A)((u)), b+uB\to \Omega(\Rep(A))[[\Gg]]^G((u)), ud+\iota_x$$
\end{prop}
Here $\iota_x$ refers to the contraction of a form by a vector field corresponding to $x$.
(Note that the value of this map at $x=0$ is the HKR morphism).
\begin{proof} By Proposition \ref{prop:chevalley}, taking into account that $[d,d]=0$ and $d(x)=0,$
$${\rm{HKR}}(x)((b+uB)(c))=ud{\rm{HKR}}(x)(c)+\sum _{K=0}^\infty \frac{1}{K!}\chi (\delta(x) (d+x)^{K})(c).$$
If we put $c=\alpha_0 \otimes \ldots \otimes \alpha_n,\; \alpha _i\in {\mathcal A},$ then the second summand is equal to
$$\sum _{N=0}^\infty \sum _{\sum _0^n N_k =N} \sum _{i=1}^N {\rm{tr}}(\alpha_0 x^{N_0}  d\alpha_1 x^{N_1} \ldots [x,\alpha_i]x^{N_i}\ldots d\alpha_n x^{N_n})$$
Now observe that for $x\in \Gg,$ if $L_x {\rm{ev}}(a)=[x,{\rm{ev}}(a)]$ where $L_x$ denotes the Lie derivative by the vector field on ${\rm{Rep}}(A)$ corresponding to $x$. Therefore, for $\alpha_i={\rm{ev}}(a_i),$ the above formula turns into
$$\sum _{N=0}^\infty \sum _{\sum _0^n N_k =N} \sum _{i=1}^N {\rm{tr}}(\alpha_0 x^{N_0}  d\alpha_1 x^{N_1} \ldots L_x\alpha_i x^{N_i}\ldots d\alpha_n x^{N_n})=$$
$$\iota_x \sum _{N=0}^\infty \sum _{\sum _0^n N_k =N} {\rm{tr}}(\alpha_0 x^{N_0}  d\alpha_1 x^{N_1} \ldots d\alpha_n x^{N_n})=\iota _x {\rm{HKR}}(x)(c).$$
\end{proof}
\begin{remark}\label{rmk:hochschild hkr}
As for the Hochschild homology, the fact that the map
$${\rm{HKR}}:C_{-\bullet}(A)\to \Omega(\Rep(A))$$
(the value of ${\rm{HKR}}(x)$ at $x=0$) sends Hochschild cycles to basic forms follows from the formula (used in the proof above) $\iota _x {\rm{HKR}}(x)(c)={\rm{HKR}}(x\iota_\Delta (c))$, and the fact that $\iota_\Delta$ kills the image of $b$.
\end{remark}
\subsection{More on noncommutative calculus}\label{ss:more on nc calc} Let us finish by saying a few words about noncommutative analogues of formulas \eqref{eq:Cartan 1} and other algebraic properties of forms on manifolds. Note first that operators $\iota_X$ and $L_X$ can be defined not just for vector fields and functions but for multivector fields; they satisfy \eqref{eq:Cartan 1}, the bracket being the Schouten-Nijenhuis bracket. Since all the equations \eqref{eq:Cartan 1} are in terms of commutators, they can be interpreted as an action of certain differential graded Lie algebras on the complex of forms. This latter formulation has a noncommutative analogue as follows. For any (graded) associative algebra $A$, let $\Gg (A)$ be the graded Lie algebra of its Hochschild cochains, with the Gerstenhaber bracket and the Hochschild differential $\delta$. Then on $(C_{-\bullet}(A)[[u]], b+uB)$ there is a $k[[u]]$-linear, $(u)$-adically continuous structure of an $L_\infty$ module over $(\Gg (A)[\epsilon, u], \delta+u\frac{\partial}{\partial \epsilon}),$ such that, for a cochain $D$ of degree $\leq 1,$ $D\epsilon $ acts by $-\iota _D$ as in \eqref{eq: iota hoch 1}, \eqref{eq: iota hoch 2}. From that, one deduces a construction of a flat superconnection on the periodic cyclic complex of a family of algebras (\cite{DTT}, Proposition 2). The formulas for the $L_\infty$ action are somewhat more complicated than the definition of $\chi$ in \eqref{eq:chi}. It would be interesting to better understand the relation between the above and the construction of the Gauss-Manin connection in the present paper.

Let us finish by mentioning one more feature of the classical calculus. The space of multivector fields is in fact a Gerstenhaber algebra, in particular a graded commutative algebra; one has 
\begin{equation}\label{eq:calc}
\iota_X\iota_Y=\iota_{XY};\; L_{XY}=(-1)^{|Y|}L_X\iota_Y+\iota_XL_Y.
\end{equation}
An algebraic system uniting \eqref{eq:Cartan 1} and \eqref{eq:calc} is called a calculus; it was shown in \cite{KS2}, \cite{TT}, \cite{DTT1} that, for an associative algebra $A$, the pair of complexes $(C^*(A,A), C_*(A,A))$ is always a homotopy calculus. This structure quite inexplicit and not canonical; it depends on a choice of a Drinfeld associator.

{\small{
\bibliographystyle{amsalpha}
\providecommand{\bysame}{\leavevmode\hbox to3em{\hrulefill}\thinspace}
\providecommand{\MR}{\relax\ifhmode\unskip\space\fi MR }
\providecommand{\MRhref}[2]{%
  \href{http://www.ams.org/mathscinet-getitem?mr=#1}{#2}
}
\providecommand{\href}[2]{#2}

}}

\smallskip

{\footnotesize
\noindent
{{\textbf{V.G.}: Department of Mathematics, University of Chicago, 5734 S. University Ave,
Chicago, IL
60637, USA;}\newline
\hphantom{x}\quad\, {\tt ginzburg@math.uchicago.edu}} \\
{\bf T.S.}: Department of Mathematics, MIT, Cambridge, MA 02139, USA.
  \hphantom{x}\quad\, {\tt trasched@gmail.com}}

\end{document}